\documentclass[12pt]{amsart}
\usepackage{amsmath,amsthm}
\usepackage{amssymb,latexsym,eufrak,amsmath,amscd,graphicx}
\usepackage[mathscr]{eucal}
  \usepackage[all]{xy}
\theoremstyle{plain}
\newtheorem{theorem}{Theorem}[section]

\newtheorem{proposition}[theorem]{Proposition}
\newtheorem{corollary}[theorem]{Corollary}
\newtheorem{lemma}[theorem]{Lemma}

\newtheorem*{propositionn}{Proposition}

\newtheorem{propdef}[theorem]{Proposition/Definition}

\numberwithin{equation}{section}

\theoremstyle{definition}
\newtheorem{definition}[theorem]{Definition}

\newtheorem{remark}[theorem]{Remark}
\newtheorem{example}[theorem]{Example}
\newtheorem{exercise}[theorem]{Exercise}

\newcommand{\lra}{\longrightarrow}
\newcommand{\noi}{\noindent}
\newcommand{\PP}{\mathbf{P}}

\newcommand{\RR}{\mathbf{R}}

\newcommand{\NN}{\mathbf{N}}

\newcommand{\CC}{\mathbf{C}}
\newcommand{\QQ}{\mathbf{Q}}
\newcommand{\Adj}[2]{\textnormal{Adj}_{#1}({#2})}

\newcommand{\OO}{\mathcal{O}}

\newcommand{\II}{\mathcal{I}}
\newcommand{\FF}{\mathcal{F}}

\newcommand{\fra}{\mathfrak{a}}
\newcommand{\frb}{\mathfrak{b}}

\newcommand{\frd}{\mathfrak{d}}
\newcommand{\bull}{_{\bullet}}

\newcommand{\frakm}{\mathfrak{m}}
\newcommand{\frq}{\mathfrak{q}}

\newcommand{\frj}{\mathfrak{j}}
\newcommand{\frr}{\frak{r}}

\newcommand{\eps}{\varepsilon}
\newcommand{\sbl}{\vskip 3pt}
\newcommand{\lbl}{\vskip 6pt}

\newcommand{\HH}[3]{H^{{#1}} \big( {#2} , {#3}
\big) }
\newcommand{\HHH}[3]{H^{{#1}} \Big( {#2} , {#3}
\Big) }
\newcommand{\hh}[3]{h^{{#1}} \big( {#2} , {#3}
\big) }

\newcommand{\lct}{\textnormal{lct}}

\newcommand{\exc}{\text{Exc}}
\newcommand{\MI}[1]{\mathcal{J}  ( {#1}
) }

\newcommand{\ord}{\textnormal{ord}}
\newcommand{\codim}{\textnormal{codim}}
\newcommand{\mult}{\textnormal{mult}}
\newcommand{\pr}{\prime}

\newcommand{\num}{ \equiv_{\text{num}} }
\newcommand{\lin}{\equiv_{\text{lin}} }

\newcommand{\Bl}{\text{Bl}}
\newcommand{\divisor}{\textnormal{div}}

\newcommand{\Linser}[1]{| \mspace{1.5mu} {#1}
\mspace{1.5mu} |}
\newcommand{\linser}[1]{\Linser{  {#1}  }}
\newcommand{\alinser}[1]{\| \mspace{1.5mu}{#1}
		\mspace{1.5mu}\|}

\newcommand{\bs}[1]{\frb \big(\mspace{1.1mu}
\linser{{#1}} \mspace{1.1mu} \big)}

\newcommand{\tn}[1]{\textnormal{{#1}}}

 \newcommand{\interior}{\textnormal{int}}
\newcommand{\BBB}{\mathbf{B}}
\newcommand{\MIan}[1]{\mathcal{J}_{\textnormal{an}}{( {#1} )}}

\begin{document}
  
\title[A Short Course on Multiplier Ideals]{A Short Course  on Multiplier Ideals}

\author{Robert Lazarsfeld}
\address{Department of Mathematics, University of Michigan,  Ann Arbor, MI  48109}
\email{rlaz@umich.edu}
\thanks{Research partially supported by NSF grant DMS 0652845}

\maketitle
\setlength{\parskip}{3pt }
\tableofcontents
  
\setlength{\parskip}{.1 in}

\section*{Introduction}

These notes are the write-up of my 2008 PCMI lectures on multiplier ideals. 
They aim to give an introduction to the algebro-geometric side of the theory, with an emphasis on its global aspects. Besides serving as warm-up for the lectures of Hacon, my hope was to convey to the audience a feeling for the sorts of problems for which multiplier ideals have proved useful. Thus I have focused on concrete examples and applications at the expense of general theory. While referring to \cite{PAG} and other sources for some  technical points, I have tried to include sufficient detail here so that the conscientious reader can arrive at a reasonable grasp of the machinery by working through these lectures. 

The revolutionary work of Hacon--McKernan, Takayama and Birkar--Cascini--Hacon--McKernan  (\cite{HM}, \cite{HM2}, \cite{Tak}, \cite{BCHM}) appeared shortly after the publication of \cite{PAG}, and these papers have led to some changes of perspectives on multiplier ideals.   In particular, the first three made clear the importance of adjoint ideals as a tool in proving extension theorems; these were not so clearly in focus at the time \cite{PAG} was written.  I have taken this new viewpoint into account in discussing  the restriction theorem in Lecture 3. Adjoint ideals also open the door to an extremely transparent presentation of Siu's theorem on deformation-invariance of plurigenera of varieties of general type, which appears in Lecture 5.

Besides Part III of \cite{PAG}, I have co-authored an overview of multiplier ideals once before, in \cite{BL}. Those notes focused more on the local and algebraic aspects of the story.  The analytic theory is  surveyed in \cite{SiuMultiplier}, as well as in other lecture series in this volume. 

I wish to thank Eugene Eisenstein, Christopher Hacon, J\'anos Koll\'ar and Mircea Musta\c t\u a    for valuable suggestions. I am particularly grateful to  Sam Grushevsky, who read through in its entirety a draft of these lectures, and made copious suggestions. 

\section{Construction and Examples of Multiplier Ideals}
This preliminary lecture is devoted to the construction and first properties of multiplier ideals. We start by discussing the  algebraic and analytic incarnations of these ideals. After giving the example of monomial ideals, we survey briefly some of the invariants of singularities that can be defined via multiplier ideals. 

\subsection*{Definition of Multiplier Ideals}
In this section, we will give the definition of multiplier ideals.

We work throughout with a smooth algebraic variety $X$ of dimension $d$ defined over $\CC$. For the moment, we will deal with two sorts of geometric objects on $X$:   an ideal sheaf $\fra \subseteq \OO_X$ together with a weighting coefficient $c > 0$, and an effective $\QQ$-divisor  $D$ on $X$. Recall that the latter consists of a formal linear combination  
\[ D \ = \ \sum a_i D_i, \]
where the $D_i$ are distinct prime divisors and each $a_i \in \QQ$ is a non-negative rational number. We will attach to these data multiplier ideal sheaves 
\[
\MI{\fra^c}\  \ , \  \MI{D} \ \subseteq \ \OO_X. \]
The intuition is that these ideals will measure the singularities of $D$  or of functions $f \in \fra$, with ``nastier" singularities being reflected in ``deeper" multiplier ideals. 

Although we will mainly focus on algebraic constructions, it is perhaps most intuitive to start with the analytic avatars of multiplier ideals. 
\begin{definition} [{Analytic multiplier ideals}]
  Given $D=\sum a_i D_i$ as above, choose local equations $f_i \in \OO_X$ for each $D_i$. Then the (analytic) multiplier ideal of $D$ is given locally by
\[
\MIan{X,D}\ =_{\text{locally}} \ \Big\{ \, h \in \OO_X \, \Big | \ \frac{ |h|^2}{\prod |f_i|^{2a_i}} \text{ is locally integrable } \, \Big \} . \]
Similarly, if $f_1, \ldots , f_r \in \fra$ are local generators, then
\[
\MIan{X,\fra^c}\ =_{\text{locally}} \ \Big\{ \, h \in \OO_X \, \Big | \ \frac{ |h|^2}{\big(\sum |f_i|^{2}\big)^c} \text{ is locally integrable } \, \Big \} .   \]
(One checks that these do not depend on the choice of the $f_i$.)\qed
\end{definition}
  Equivalently, $\MIan{D}$ and $\MIan{\fra^c}$ arise as the multiplier ideal $\MI{\phi}$, where $\phi$ is the appropriate one of the two plurisubharmonic functions
\[
\phi \ = \ \sum  \log |f_i|^{2a_i}  \ \  \text{or} \ \ \phi \ = \ c \cdot \log \big(\sum |f_i|^2 \big). \]
Note that this construction exhibits quite clearly the yoga that ``more" singularities give rise to ``deeper" multiplier ideals: the singularities of $f \in \fra$ or of $D$ are reflected in the rate at which the real-vallued functions
\[
\frac{1}{\prod |f_i|^{2a_i}} \ \ \text{or} \ \ \frac{1}{  \big(\sum |f_i|^{2}\big)^c}
\]
blow-up along the support of $D$ or the zeroes of $\fra$, and this  in turn is measured by the vanishing of the multipliers $h$ required to ensure integrability.

\begin{exercise}  \label{An.MI.SNC.Divisor}  Suppose that $ D = \sum a_i D_i$ has simple normal crossing support. Then 
  \[ 
  \MIan{X, D} \ = \ \OO_X( - [D] ),  \]  
  where $[D] = \sum [a_i] D_i$ is the round-down (or integer part) of $D$. 
  (\textsc{Hint}:  This boils down to the assertion that if $z_1, \ldots, z_d$ are the standard complex coordinates in $\CC^d$, and if $h \in \CC\{ z_1, \ldots, z_d \}$ is a convergent power series, then 
  \[  \frac{|h|^2}{|z_1|^{2a_1} \cdot \ldots \cdot |z_d|^{2 a_d}}
  \]
  is locally integrable near the origin if and only if
  \[
  z_1^{[a_1]} \cdot \ldots \cdot z_d^{[a_d]} \ | \ h 
  \]
 in $\CC\{ z_1, \ldots, z_d \}$. By separating  variables, this in turn follows from the elementary computation that  the function  $1/|z|^{2c}$ of one variable is locally integrable if and only if $c < 1$.) \qed  \end{exercise}

 Multiplier ideals can also be constructed algebro-geometrically. 
 Let 
 \[ \mu : X^\pr \lra X \]
 be a log resolution of $D$ or of $\fra$. Recall that this means to begin with that $\mu$ is a proper morphism, with $X^\pr$ smooth. In the first instance we require that $\mu^* D + \text{Exc}(\mu)$ have simple normal crossing (SNC) support, while in the second one asks that
 \[ \fra \cdot \OO_{X^\pr} \ = \ \OO_{X^\pr}(-F) \]
 where $F$ is an effective divisor and  $F + \text{Exc}(\mu)$ has SNC support. 
 We consider also the relative canonical bundle
 \[  K_{X^\pr / X} \ = \ \det ( d\mu ), \]
 so that $K_{X^\pr / X} \lin K_{X^\pr} - \mu^* K_X$. Note that this is well-defined as an actual divisor supported on the exceptional locus of $\mu$ (and not merely as a linear equivalence class).  
 \begin{definition} [Algebraic multiplier ideal]  \label{Def.Alg.Mult.Ideals}
 The multiplier ideals associated to $D$ and to $\fra$ are defined to be:
 \begin{align*}
 \MI{D} \ &= \ \mu_*  \OO_{X^\pr}(K_{X^\pr/X} - [\mu^*D ] ) \\ 
  \MI{\fra^c} \ &= \ \mu_*  \OO_{X^\pr}(K_{X^\pr/X} - [cF ] ).
   \end{align*}
 (As in the previous Exercise, the integer part of a $\QQ$-divisor is defined by taking the integer part of each of its corefficients.) \qed \end{definition} 
 \noi Observe that these are subsheaves of 
 \[ \mu_* \OO_{X^\pr}(K_{X^\pr/X}) \ = \ \OO_X, \]
 i.e. they are indeed ideal sheaves.
 
 One can rephrase the definition more concretely in terms of discrepancies. Write
 \begin{equation} \label{Discrepancy.Equation}
 \mu^* D \ = \ \sum r_i E_i \ \ , \ \ K_{X^\pr/X} \ = \ \sum b_i E_i, \end{equation}
 where the $E_i$ are distinct prime divisors on $X^\pr$:  thus the $r_i$ are non-negative rational numbers and the $b_i$ are non-negative integers. We view each of the $E_i$ as defining a valuation $\ord_{E_i}$ on rational or regular functions on $X$. Then it follows from  Definition \ref{Def.Alg.Mult.Ideals} that 
  \[
  \MI{D} \ = \ \big \{ f \in \CC(X) \, \big | \, \ord_{E_i}(f) \ge [r_i] - b_i \, ,  
 \, \text{with $f$ otherwise regular} \big \},   \]
with a similar interpretation of $\MI{\fra^c}$. (Note that we are abusing notation a bit here: $\MI{D}$ is actually the sheaf determined in the evident manner by the recipe on the right.) Observe that $b_i > 0$ only when $E_i$ is $\mu$-exceptional, so the condition \[  \ord_{E_i}(f) \ge [r_i] - b_i \] does not allow $f$ to have any poles on $X$. Thus we see again that $\MI{D}$ is a sheaf of ideals.

\begin{remark}
The definitions of 
  $\MIan{D}$ and $\MI{D}$ may seem somewhat arbitrary or unmotivated, but they  are actually dictated by the vanishing theorems that multiplier ideals satisfy. In the algebraic case, this will become clear for example in the proof of Theorem \ref{Nadel.Vanishing.Theorem}. \qed
\end{remark}

\begin{example}
We work out explicitly one (artificially) simple example. Let $X = \CC^2$, let $A_1, A_2, A_3 \subseteq X$ be three distinct lines through the origin, and set
\[  D \ = \ \frac{2}{3} ( A_1 + A_2 + A_3 ). \]
Then $D$ is resolved by simply blowing up the origin:
\[ \mu : X^\pr \, = \, \Bl_0(X) \lra X.
\]
Writing $E$ for the exceptional divisor of $\mu$, and $A_i^\pr$ for the proper transform of $A_i$, one has
\begin{align*} 
\mu^* (A_1 + A_2 + A_3) \ &= \ (A_1^\pr + A_2^\pr + A_3^\pr) \, + \, 3E \\
\mu^* D \ &= \ \frac{2}{3} ( A_1^\pr + A_2^\pr + A_3^\pr) + 2E \\
[\mu^* D] \ &= \ 2E. \end{align*}
Moreover $K_{X^\pr/X} = E$, and hence
\[  
\MI{D} \ = \ \mu_* \OO_{X^\pr}(-E)
\]
is the maximal ideal of functions vanishing at the origin. Observe that this computation also shows that rounding does not in general commute with pull-back of $\QQ$-divisors. \qed
\end{example}

The algebraic construction of multiplier ideals started by choosing a resolution of singularities. Therefore it is important to establish:
\begin{proposition}
The multiplier ideals $\MI{D}$ and $\MI{\fra^c}$ do not depend on the resolution used to construct them.
\end{proposition}
\noi In brief, using the fact that any two resolutions can be dominated by a third, one reduces to checking that if $X$ is already a log resolution of the data at hand, then nothing is changed by passing to a further blow-up:
  \begin{lemma}
Assume that $D$ has SNC  support, and let $\mu : X^\pr \lra X$ be a further log resolution of $(X, D)$. Then
\[
\mu_* \OO_{X^\pr}\big(K_{X^\pr/X} - [\mu^*D]\big) \ = \ \OO_X\big(-[D]\big).
\]
 \end{lemma}
 \noi This in turn can be checked by an elementary direct calculation. We refer to \cite[9.2.19]{PAG} for details. 
 
 The next point is to reconcile the analytic and algebraic constructions of multiplier ideals. 
 \begin{proposition}  \label{Analy.vs.Alg.MI}
 Let $D$ be an effective $\QQ$-divisor on $X$. Then 
 \[
 \MIan{X,D} \ = \ \MI{X, D}, \]
 and similarly $\MIan{X, \fra^c} = \MI{X,\fra^c}$ for any ideal sheaf $\fra$.   \end{proposition}
 \noi (Strictly speaking, the analytic multiplier ideals are the analytic sheaves associated to their algebraic counterparts, but we do not dwell on this distinction.)  
 
 For the Proposition, the key point is that both species of multiplier ideals transform the same way under birational morphisms:
 \begin{lemma} \label{Birat.Transf.Rule}
 Let $\mu : X^\pr \lra X$ be a proper birational map, and let $D$ be an effective $\QQ$-divisor on $X$. Then:
 \begin{align*}
 \MIan{X,D} \ &= \ \mu_* \big(  \OO_{X^\pr}(K_{X^\pr/X})\otimes \MIan{X^\pr, \mu^*D} \big) \\
  \MI{X,D} \ &= \ \mu_* \big( \OO_{X^\pr}(K_{X^\pr/X}) \otimes \MI{X^\pr, \mu^*D}  \big) . 
  \end{align*}
 \end{lemma}
 \noi  
 In the analytic setting this is a consequence of the change of variables formula for integrals, while the algebraic statement is established with a little computation via the projection formula. The Proposition
 follows at once from the Lemma. In fact, one is reduced to proving Proposition  \ref{Analy.vs.Alg.MI}
 when $D$ or $\fra$ are already in normal crossing form, and this case is handled by Exercise \ref {An.MI.SNC.Divisor}.
 
 \begin{remark}[Multiplier ideals on singular varieties] 
 \label{Mult.Ideal.Sing.Var.Remark}
 Under favorable circumstances, Definition \ref{Def.Alg.Mult.Ideals}
makes sense even when $X$ is singular. The main point at which non-singularity is used in the discussion above is  to be able to define the relative canonical bundle $ K_{X^\pr/X} = K_{X^\pr} - \mu^* K_X$ of a log resolution \[\mu : X^\pr \lra X\] of $(X,D)$. For this it is enough that $X$ is normal and that $K_{X}$ is Cartier or even $\QQ$-Cartier, so that $\mu^* K_X$ is defined. Thus Definition  \ref{Def.Alg.Mult.Ideals} goes through without change provided that $X$ is Gorenstein or $\QQ$-Gorenstein.  For an  arbitrary normal variety $X$, one can introduce a ``boundary" $\QQ$-divisor $\Delta$ such that $K_X + \Delta$ is $\QQ$-Cartier, and define multiplier ideals 
\[ \MI{(X, \Delta);D}  \ \subseteq \ \OO_X. \] These generalizations are discussed briefly in \cite[\S 9.3.G]{PAG}. DeFernex and Hacon explore in \cite{DeFH} the possibility of defining multiplier ideals (without boundaries) on an arbitrary normal variety. However in the sequel we will work almost exclusively with smooth ambient varieties $X$. \qed  \end{remark}

 We conclude this section with two further exercises for the reader.
 
 \begin{exercise}
 Assume that $X$ is affine, and let $\fra \subseteq \CC[X]$ be an ideal. Given $c > 0$, choose $k > c$ general elements
\[  f_1, \ldots, f_k \ \in \ \fra, \]
let $A_i = \divisor(f_i)$, and put 
$D = \frac{c}{k}  ( A_1 + \ldots + A_k )$. Then
\[ \MI{D} \ = \ \MI{\fra^c}. \]
(By a ``general element" of an ideal, one means a general $\CC$-linear combination of a collection of generators of the ideal.) \qed
\end{exercise}
 
 \begin{exercise} \label{Large.Mult.Divisors}
 Let $D = \sum a_i D_i$ be an effective $\QQ$-divisor on $X$.
 Assume that
 \[ \mult_x(D) \ =_{\text{def}} \ \sum a_i \cdot \mult_x(D_i) \ \ge \ \dim X \]
 for some point $x \in X$ . Then $\MI{X, D}$ is non-trivial at $x$, i.e.
 \[  \MI{D}_x \ \subseteq \ \frakm_x \ \subseteq \OO_xX,\]
 where $\frakm_x \subseteq \OO_X$ is the maximal ideal of $x$. 
 (Compute the multiplier ideal  in question using a resolution $\mu : X^\pr \lra X$  that dominates the blow-up of $X$ at $x$, and observe that
 \[  \ord_E\big(K_{X^\pr/X} - [\mu^*D]\big) \ \le \ -1,\]
 where $E$ is the proper transform of the exceptional divisor over $x$.)  \qed\end{exercise}

 \subsection*{Monomial Ideals}
 
It is typically very hard to compute the multiplier ideal of an explicitly given divisor or ideal. One important class of examples that has been worked out is that of monomial ideals on affine space. These are handled by a theorem of Howald \cite{Howald}.

Let $X = \CC^d$, and let \[ \fra \ \subseteq \ \CC[x_1, \ldots, x_d ]\] be an ideal generated by monomials in the $x_i$. Observe that such a monomial is specified by an exponent vector $w = (w_1, \ldots, w_d ) \in \NN^d$: we write
\[  x^w \ = \ x_1^{w_1} \cdot \ldots \cdot x_d^{w_d}.
\] 
  The \textit{Newton polyhedron} 
\[ P(\fra) \ \subseteq \ \RR^d \]
of $\fra$ is the closed convex set spanned by the exponent vectors of all monomials in $\fra$. This is illustrated in Figure \ref{Newt.Polyhedron.Figure}, which shows the Newton polyhedron for the monomial ideal 
\begin{equation} \label{Monom.Ideal.Eqn}
\fra \ = \ \big\langle x^4, x^2y, xy^2, y^5 \big \rangle.
\end{equation}
Finally, put $\mathbf{1} = (1, \ldots, 1) \in \NN^d$. 

\begin{figure}\includegraphics[scale = .75]{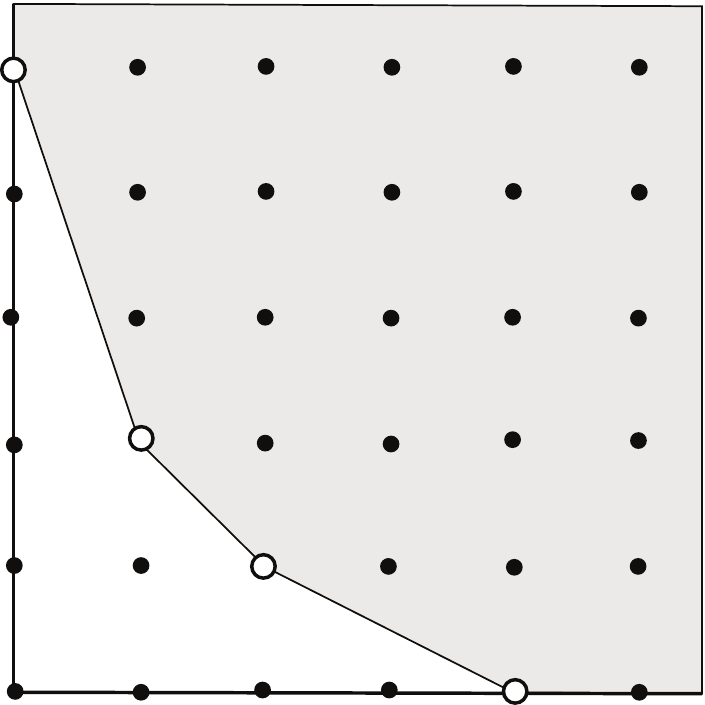}
\caption{Newton Polyhedron of $ \langle x^4, x^2y, xy^2, y^5\rangle$}
\label{Newt.Polyhedron.Figure}
\end{figure}

 Howald's statement is the following:
 \begin{theorem}
 For any $c >0$, 
 \[ \MI{\fra^c} \ \subseteq \ \CC[x_1, \ldots, x_d] \]
 is the monomial ideal spanned by all monomials $x^w$ where
 \[  w \, + \, \mathbf{1} \ \in \ \interior \big(  \, c\cdot P(\fra) \, \big ). \]
  \end{theorem}
 
 \noi Once one knows the statement for which one is aiming, the proof is relatively straight-forward: see \cite{Howald} or \cite[\S 9.3.C]{PAG}. 
 
 \begin{example} For $\fra =  \langle x^4, x^2y, xy^2, y^5\rangle$ , one has
 $\MI{\fra} = (x^2, xy, y^2)$, while if $0<c < 1$ then $(x, y) \subseteq \MI{\fra^c}$. \qed
  \end{example}
 
 \begin{example} \label{Diag.Ideal}
  Let 
 \[ \fra \ = \ \big ( x_1^{e_1}, \ldots, x_d^{e_d} \big) \ \subseteq \ \CC[x_1, \ldots, x_d ]. \]
 Writing $\xi_1, \ldots, \xi_d$ for the natural coordinates on $\RR^d$ adapted to $\NN^d \subseteq \RR^d$, the Newton polyhedron  $P(\fra) \subseteq \RR^d$ of $\fra$ is the region in the first orthant given by the equation 
 \[  \frac{\xi_1}{e_1} \, + \, \ldots \, + \, \frac{\xi_d}{e_d}\  \ge \ 1. \]
 Hence $\MI{\fra^c}$ is the monomial ideal spanned by all monomials $x^w$ whose exponent vectors satisfy the equation
 \[ \frac{w_1+1}{e_1} \, + \, \ldots \, + \, \frac{w_d+1}{e_d}\ > \ c.\ \qed \]
 \end{example}

 \subsection*{Invariants Defined by Multiplier Ideals}
 
 Multiplier ideals lead to invariants of the singularities of a divisor or the functions in an ideal. The most important and well-known is the following:
  
 \begin{definition}[Log-canonical threshold] \label{Def.LCT}
 Let $D$ be an effective $\QQ$-divisor on a smooth variety $X$, and let $x \in X$ be a fixed point. The \textit{log-canonical threshold} of $D$ at $x$ is 
 \[
\lct_x(D) \ =_{\text{def}} \ \min \big\{  \, c > 0 \, \big | \, \MI{cD} \ \text{ is non-trivial at $x$ } \big \}.
\]
The log-canonical threshold of an ideal sheaf $\fra \subseteq \OO_X$ is defined analogously. \qed
 \end{definition}
 \noi 
  Thus small values of $\lct_x(D)$ reflect more dramatic singularities. 
 Definition \ref{KLT.LC.Def} below explains  the etymology of the term.
 
 \begin{exercise}
Let $\mu : X^\pr \lra X$ be a log resolution of $D$, and as in equation \eqref {Discrepancy.Equation} write
\[ 
\mu^* D \ = \ \sum r_i E_i \ \ , \ \ K_{X^\pr / X} \ = \ \sum b_i E_i. \]
Then
\[
\lct_x(D)  \ = \ \min_{ \mu(E_i) \owns x} \Big \{  \frac{ b_i + 1}{r_i} \Big \},
\]
the minimum being taken over all $i$ such that $x$ lies in the image of the corresponding exceptional divisor. In particular, $\lct_x(D)$  is rational.
 \end{exercise}
 
 \begin{example} [Complex singularity exponent] One of the early appearances of the log-canonical threshold was in the work \cite{Var} of Varchenko,  who studied the \textit{complex singularity exponent} $c_0(f)$ of a polynomial or holomorphic function $f$ in a neighborhood of $0 \in \CC^d$. Specifically, set
 \begin{equation}
 c_0(f) \ = \ \sup \Big \{ c > 0 \, \Big | \, \frac{1}{|f|^{2c}} \ \text{ is locally integrable near $0$ } \Big \}. \tag{*} \end{equation}
 Writing $\lct_0(f)$ for the log-canonical threshold of the divisor determined by $f$ (or equivalently of the principal ideal generated by $f$), it follows from Proposition \ref {Analy.vs.Alg.MI} that
 \[
 c_0(f) = \lct_0(f). 
 \] 
In view of the previous exercise, this establishes the fact -- which is certainly not obvious from (*) -- that $c_0(f)$ is rational. (The rationality of $c_0(f)$ was proven in this manner by Varchenko, although his work pre-dates the language of multiplier ideals.) \qed
\end{example}
 
 \begin{exercise}
 If $D = \{ x^3 - y^2 = 0 \} \subseteq \CC^2$, then $\lct_0(D) = \frac{5}{6}$. \qed\end{exercise}
 
  \begin{exercise}
 Consider as in Example \ref{Diag.Ideal} the monomial ideal
 \[  \fra \ = \ (x_1^{e_1} , \ldots, x_d^{e_d} ) \ \subseteq  \ \CC[x_1, \ldots, x_d]. \]
 Then $\lct_0(\fra) = \sum \frac{1}{e_i}$. \qed
 \end{exercise}

 The log-canonical threshold is the first of a sequence of invariants defined by the ``jumping" of multiplier ideals. Specifically, observe that the ideals $\MI{cD}$ become deeper as the coefficient $c$ grows. So one is led to:
 \begin{propdef} 
 \label{Propdef.Jumping.Nos}
 In the situation of Definition \ref{Def.LCT}, there exists a discrete sequence of rational numbers $\xi_i = \xi_i(D;x)$ with
  \[
  0 = \xi_0 \ < \ \xi_1 \ < \ \xi_2 \ < \ \ldots
\]
characterized by the property that \tn{(}the stalks at $x$ of\tn{)} the multiplier ideals $\MI{cD}_x$ are constant exactly for 
\[  c \ \in \ [\xi_i , \xi_{i+1}). \]
The $\xi_i$ are called the \textit{jumping numbers} of $D$ at $x$.
\end{propdef}
\noi Jumping numbers of an ideal sheaf $\fra \subseteq \OO_X$ are defined similarly.

  It follows from the definition that $\lct_x(D) = \xi_1(D;x)$. In  the notation of \eqref {Discrepancy.Equation}, the $\xi_i$ occur among the rational numbers ${(b_i + m)}/{r_i}$ for various $m \in \NN$.  First appearing implicitly in work of Libgober and Loeser-Vaqui\'e, these quantities were studied systematically in \cite{JCMI}. In particular, this last paper establishes some connections between jumping coefficients and other invariants.

\begin{exercise} Compute the jumping numbers of the ideal $\fra = (x_1^{e_1}, \ldots, x_d^{e_d})$. \qed
\end{exercise}
 
 \begin{exercise} \label{Mustata.Observation} Let $ \xi_i < \xi_{i+1}$ be consecutive jumping coefficients of an ideal $\fra \subseteq \OO_X$ at a point $x \in X$. Then  \[   \sqrt{\fra} \cdot \MI{\fra^{\xi_i}}_x\ \subseteq \ \MI{\fra^{\xi_{i+1}}}_x
\] 
in $\OO_xX$. In particular,
\[
(\sqrt{\fra})^m \ \subseteq \ \MI{\fra^{\xi_m}}_x
\]
for every $m > 0$. (This was pointed out to us by M. Musta\c t\u a. See Example \ref{Skoda.Nullstellensatz} for an application.) \qed
 \end{exercise}
 
 \begin{remark} \label{Global.LCT.JN}
 In an analogous fashion, one can define the log-canonical threshold $\lct(D)$  and $\lct(\fra)$, as well as  jumping numbers $\xi_i(D)$ and $\xi_i(\fra)$,  globally on $X$, without localizing at a particular point. We leave the relevant definitions -- as well as the natural extension of the previous Exercise --  to the reader. \qed \end{remark}
 
  Finally, we note that multiplier ideals lead to some natural classes of singularities for a pair $(X,D)$ consisting of a smooth variety $X$ and an effective $\QQ$-divisor $D$ on $X$.
 \begin{definition} \label{KLT.LC.Def}
 One says that $(X,D)$ is \textit{Kawamata log-terminal} (KLT) if
 \[  \MI{X,D} \ = \ \OO_X. \]
 The pair $(X,D)$ is \textit{log-canonical} if 
 \[  \MI{X, (1-\eps) D} \ = \ \OO_X \]
 for $0 < \eps \ll 1$. \qed
   \end{definition}
 \noi These concepts (and variants thereof) play an important role in the minimal model program, although in that setting one does not want to limit oneself to smooth ambient varieties. 
 
 \begin{remark} [Characteristic $p$ analogues] Work of Smith, Hara, Yoshida, Watanabe, Takagi, Musta\c t\u a and others has led to the development of theory in characteristic $p > 0$ that closely parallels the theory of multiplier ideals, and reduces to it for ideals lifted from characteristic $0$. We refer to the last section of \cite{Ein.Mustata} for a quick overview and further references. 
 \qed
  \end{remark}

 
\section{Vanishing Theorems for Multiplier Ideals}

In this Lecture we discuss the basic vanishing theorems for multiplier ideals, and give some first applications. As always, we work with varieties over $\CC$.

\subsection*{The Kawamata--Viehweg--Nadel Vanishing Theorem}

We start by recalling some definitions surrounding positivity for divisors. Let $X$ be an irreducible projective variety of dimension $d$, and let $B$ be a (Cartier) divisor on $X$.
One says that $B$ is \textit{nef} (or \textit{numerically effective}) if 
\[   ( B \cdot C  ) \ \ge \ 0 \ \ \text{ for every irreducible curve $C \subseteq X$}.
\]
Nefness means in effect that $B$ is a limit of ample divisors: see \cite[Chapter 1.4]{PAG} for a precise account. A divisor $B$ is \textit{big} if the spaces of sections of $mB$ grow maximally with $m$,  i.e. if
\[
\hh{0}{X}{\OO_X(mB) } \ \sim  \ m^{d} \ \ \text{ for $m \gg 0$. } \]
These definitions extend in the evident manner to $\QQ$-divisors.
We will first deal with divisors that are both nef and big: a typical example   arises by pulling back an ample divisor under a birational morphism.  

A basic fact is that for a nef divisor, bigness is tested numerically:
\begin{lemma}
Assume that $B$ is nef. Then $B$ is big if and only if its top self-intersection number is strictly positive: $(B^d) > 0$. 
\end{lemma}
\noi See \cite[\S 2.2]{PAG} for a proof. This Lemma also remains valid for $\QQ$-divisors.

The fundamental result for our purposes was proved independently by Kawamata and Viehweg in the early 1980's.
\begin{theorem}[Kawamata--Viehweg Vanishing Theorem] \label{Kawamata.Viehweg.Vanishing}
Let $X$ be a smooth projective variety. Consider an integral divisor $L$ and an effective $\QQ$-divisor $D$ on $X$. Assume that
\begin{itemize}
\item[(i).] $L - D$ is nef and big; and
\item[(ii).]  $D$ has simple normal crossing support.
\end{itemize}
Then
\[ \HH{i}{X}{\OO_X(K_X + L - [D])} \ = \ 0 \ \ \text{ for }\ i > 0. \]
\end{theorem}
\noi As in the previous Lecture, the integer part (or round-down) $[D]$ of a $\QQ$-divisor $D$ is obtained by taking the integer part of each of its coefficients.

When $D = 0$ and $L$ is ample, this is the classical Kodaira vanishing theorem. Still taking $D = 0$, the Theorem asserts in general that the statement of Kodaira vanishing remains true for divisors that are merely big and nef: this very useful fact -- also due to Kawamata and Viehweg -- completes some earlier results of Ramanujam, Mumford, and Grauert--Riemenschneider. However the real power (and subtlety) of Theorem \ref{Kawamata.Viehweg.Vanishing} lies in the fact that while the positivity hypothesis is tested for a $\QQ$-divisor, the actual vanishing holds for a round of this divisor. As we shall see, this apparently technical improvement vastly increases the power of the result: taking integer parts can significantly change the shape of a divisor, so in favorable circumstances one gets a vanishing for divisors that are far from positive.

The original proofs of the theorem proceeded by using covering constructions to reduce to the case of integral divisors. An account of this approach, taking into account simplifications introduced by Koll\'ar and Mori in \cite{KM}, appears in \cite[\S  4.3.A,  \S 9.1.C]{PAG}. An alternative approach was developed by Esnault and Viehweg \cite{EV}, and the $L^2$ $\overline{\partial}$-machinery gives yet another proof.  In any event, it is nowadays not substantially harder to establish Theorem \ref{Kawamata.Viehweg.Vanishing} than to prove the classical Kodaira vanishing. 

The main difficulty in applying Theorem \ref{Kawamata.Viehweg.Vanishing} is that in practice the normal crossing hypothesis is rarely satisfied directly. Given an arbitrary effective $\QQ$-divisor $D$ on a variety $X$, a natural idea is to apply vanishing on a resolution of singularities and then ``push down" to get a statement on $X$. Multiplier ideals appear inevitably in so doing, and this leads to the basic vanishing theorems for these ideals. 

There are two essential results.
\begin{theorem}[Local vanishing theorem] 
\label{Local.Van.Mult.Ideals}
Let $X$ be a smooth variety, $D$ an effective $\QQ$-divisor on $X$, and \[
\mu : X^\pr \lra X \]
a log resolution of $D$. Then
\[  R^j \mu_* \OO_{X^\pr} \big( K_{X^\pr/X} - [ \mu^* D ] \big) \ = \ 0 \]
for $j > 0$. 
\end{theorem}
\noi The analogous statement holds for higher direct images of the sheaves computing the multiplier ideals $\MI{\fra^c}$. 

\begin{theorem} [Nadel Vanishing Theorem]
\label{Nadel.Vanishing.Theorem}
Let $X$ be a smooth projective variety, and let $L$ and $D$ be respectively an integer divisor and an effective $\QQ$-divisor on $X$. Assume that $L - D$ is nef and big. Then
\[
\HH{i}{X}{\OO_X(K_X + L ) \otimes \MI{D}} \ = \ 0
\]
for $i > 0$. \end{theorem}

\begin{proof}[Proof of Theorem \ref{Nadel.Vanishing.Theorem}  (granting Theorems  \ref{Kawamata.Viehweg.Vanishing} and \ref{Local.Van.Mult.Ideals})]
Let $\mu : X^\pr \lra X$ be a log resolution of $D$, and set
\[  L^\pr \ = \ \mu^* L \ \  , \ \ D^\pr \ = \ \mu^* D. \]
Thus $L^\pr - D^\pr$ is a nef and big $\QQ$-divisor on $X^\pr$, and by construction $D^\pr$ has SNC support. Therefore Kawamata--Viehweg applies on $X^\pr$ to give
\begin{equation} \label{KV.On.Blowup.Vanishing}
\HH{i}{X^\pr}{\OO_{X^\pr}(K_{X^\pr} + L^\pr - [D^\pr])} \ = \ 0
\end{equation} 
for $i > 0$. 
Now note that 
\[
K_{X^\pr} + L^\pr - [D^\pr] \  = \ K_{X^\pr/X} - [\mu^* D] + \mu^*(K_X + L). 
\]
On the other hand, one finds using the projection formula and the definition of $\MI{D}$:
\begin{align*}
\mu_* \OO_{X^\pr}\big(K_{X^\pr/X} - [\mu^* D] + \mu^*(K_X + L)\big) &= \mu_* \OO_{X^\pr}\big(K_{X^\pr/X} - [\mu^* D]\big) \otimes\OO_X(K_X + L)\\
&=   \OO_X(K_X + L) \otimes \MI{D}.
\end{align*}
But thanks to Theorem \ref{Local.Van.Mult.Ideals}, the vanishing \eqref{KV.On.Blowup.Vanishing}
 is equivalent to the vanishing 
 \[
 \HH{i}{X}{\OO_X(K_X + L) \otimes \MI{D}} \ =\ 0
 \]
 of the direct image of the sheaf in question, as required.
 \end{proof}

 The proof of Theorem \ref{Local.Van.Mult.Ideals} is similar: one reduces to the case when $X$ is projective, and applies the result of Kawamata and Viehweg in the global setting. See \cite[\S 9.4.A]{PAG} for an account. 
 
 \subsection*{Singularities of Plane Curves and Projective Hypersurfaces}

As a first illustration, we will apply these theorems to prove some classical results about singularities of plane curves and their extensions to hypersurfaces of higher dimension. Starting with a singular hypersurface, the strategy is to build a $\QQ$-divisor having a non-trivial multiplier ideal. Then the vanishing theorems give information about the postulation of the singularities of the original hypersurface. 

 Consider to begin with  a (reduced) plane curve $C \subseteq \PP^2$ of degree $m$, and let
 \[
 \Sigma \ = \ \text{Sing}(C) ,
 \]
 considered as a reduced finite subset of the plane. Our starting point is the classical
 \begin{proposition} 
\label{Sings.of.Curves}
 The set $\Sigma$ imposes independent conditions on curves of degree $k \ge m-2$, i.e.
 \begin{equation} \label{Vanishing.Sing.Curves.Eqn}
\HH{1}{\PP^2}{\II_\Sigma(k) } \ = \ 0 \ \ \text{ for }\ k \ge m-2. 
 \end{equation}
 \end{proposition} 
 \noi Here $\II_\Sigma$ denotes the ideal sheaf of $\Sigma$. We give a proof using Nadel vanishing momentarily, but first we discuss a less familiar extension due to Zariski.
   
 Specifically,  suppose that $C \subseteq \PP^2$ has a certain number of cusps, defined in local analytic coordinates by an equation $x^3 - y^2 =0$. ($C$ may have other singularities as well.) Let 
 \[  \Xi \ = \ \text{Cusps}(C), \]
 again regarded as a reduced finite subset of $\PP^2$. Zariski \cite{Zar} proved that one gets a stronger result for the postulation of $\Xi$. In fact:
 \begin{proposition}\label{Zariski.Cusps}
 One has
 \[
 \HH{1}{\PP^2}{\II_{\Xi}(k)} \ = \ 0 \ \ \text{for} \ \ k >\tfrac{5}{6}m - 3. 
 \]
 \end{proposition}
 \noi Interestingly enough, Zariksi proved this by considering the irregularity of the cyclic cover of $\PP^2$ branched along $C$. One can see the Kawamata--Viehweg--Nadel theorem as a vast generalization of this approach.
 
 \begin{proof} [Proof of Proposition \ref{Zariski.Cusps}] This is a direct consequence of Nadel vanishing. In fact, consider the $\QQ$-divisor $D = \tfrac{5}{6}C$.  Since
the log-canonical threshold of a cusp is $= \tfrac{5}{6}$, one has $\MI{D} \subseteq \II_{\Xi}$.  But as
 $C$ is reduced the multiplier ideal $\MI{D}$ is non-trivial only at finitely many points. Thus $\II_{\Xi}/\MI{D}$ is supported on a finite set,  and therefore the map \[  
 \HH{1}{\PP^2}{\OO_{\PP^2}(k) \otimes \MI{D} } \lra  \HH{1}{\PP^2}{\OO_{\PP^2}(k) \otimes \II_\Xi }
  \] 
  is surjective for all $k$. So it suffices to prove that the group on the left vanishes for $k >\tfrac{5}{6}m - 3$. But this follows immediately from Nadel vanishing upon recalling that $\OO_{\PP^2}(K_{\PP^2}) = \OO_{\PP^2}(-3)$.
 \end{proof}
 
 \begin{proof} [Proof of Proposition \ref{Sings.of.Curves}] Here an additional trick is required in order to produce a $\QQ$-divisor whose multiplier ideal vanishes on finite set including $\Sigma = \text{Sing}(C)$. Specifically, fix $0 < \eps \ll 1$, and let $\Gamma$ be a reduced curve of degree $\ell$, not containing any components of $C$, passing through $\Sigma$.   Consider the $\QQ$-divisor
 \[
 D \ = \ (1-\eps)C \, + \, 2\eps  \Gamma. 
  \]
  This has multiplicity $\ge 2$ at each singular point of $C$, and hence $\MI{D}$ vanishes on $\Sigma$ thanks to Exercise \ref{Large.Mult.Divisors}. Moreover $\MI{D}$ is again cosupported on a finite set since no component of $D$ has coefficient $\ge 1$. 
  Therefore, as in the previous proof, it suffices to show that 
  \begin{equation}
  \HH{1}{\PP^2}{\OO_{\PP^2}(k) \otimes \MI{D}} \ = \ 0 
  \tag{*} \end{equation}
  for $k \ge m-2$. But
  \[ \deg D \ = \ (1-\eps)m + 2 \eps \ell \ < \ m+1\]
  for $\eps \ll 1$, and so (*) again follows from Nadel vanishing. 
   \end{proof}
 
 \begin{example}
 When $C$ is the union of $m$ general lines, the bound  $k \ge m-2$ in \eqref{Vanishing.Sing.Curves.Eqn} is the best possible. However  we will see in Exercise \ref{Postulation.Irred.Curves} that one can take $k \ge m-3$ when $C$ is irreducible. \qed \end{example}
 
 Finally, we present a generalization of Proposition \ref{Sings.of.Curves} to higher dimensional hypersurfaces.
 \begin{proposition} \label{Park.Woo.Thm}
 Let $S \subseteq \PP^r$ be a \tn{(}reduced\tn{)} hypersurface of degree $m\ge 3$ having only isolated singularities, and set
 \[  \Sigma \ = \ \textnormal{Sing}(S). \]
 Then $\Sigma$ imposes independent conditions on hypersurfaces of degree $\ge m(r-1) - (2r - 1)$, i.e.
 \[  \HH{1}{\PP^r}{\II_{\Sigma}(k)} \ = \ 0 \ \ \text{for  } \, k \ge m(r-1) - (2r -1). \] \end{proposition}
\noi When $r = 3$ the statement was given by Severi. The general case, as well as the proof that follows, is due to Park and Woo \cite{PW}. 

\begin{proof} [Proof of Proposition \ref{Park.Woo.Thm}] We may suppose that $r \ge 3$, in which case the hypotheses imply that $S$ is irreducible. Write
\[  \Sigma \ = \ \big \{ P_1, \ldots, P_t \big\}, \]
and denote by $\Lambda \subseteq \linser{\OO_{\PP^r}(m-1)}$ the linear series  spanned by the partial derivatives of a defining equation of $S$; observe that every divisor in $\Lambda$ passes through the points of $\Sigma$.  For each $P_i \in \Sigma$, there exists a divisor $\Gamma_i \in \Lambda$ with $\mult_{P_i}(\Gamma_i) \ge 2$.\footnote{This uses that $m \ge 3$: see \cite[Lemma 3.2]{PW}.} Then for $0 < \eps \ll 1$ and $\ell \gg 0$, set
\[
D \ = \ (1-\eps)S \, + \, \eps \cdot \sum_{i = 1}^{t} \Gamma_i \, + \, \frac{\big(r-2-\eps(t - 1)\big)}{\ell}\cdot \sum_{j=1}^{\ell} A_j,
\]
where $A_1, \ldots, A_\ell \in \Lambda$ are general divisors. As $S$ is irreducible, none of the $\Gamma_i$ or $A_j$ occur as components of $S$, and therefore $\MI{D}$ is cosupported on a finite set provided that $\eps \ll 1$ and $t \gg 0$.  One has
\[
\mult_{P_i}(D) \ \ge \ (2-2\eps) \, + \, \eps(t+ 1) \, + \, \big( (r-2) - \eps(t- 1) \big) \ \ge \ r,
\]
which guarantees that $\MI{D} \subseteq \II_\Sigma$. Moreover:
\begin{align*}
\deg(D) \ &= \ m(1 - \eps) \, + \, \eps t(m-1) \, + \, \big( (r-2) - \eps(t - 1)\big)(m-1) \\ &< \ m(r-1) - (r-2),
\end{align*}
and so the required vanishing follows from Theorem \ref{Nadel.Vanishing.Theorem}. 
\end{proof}

\subsection*{Singularities of Theta Divisors}

We next discuss a theorem of Koll\'ar concerning the singularities of theta divisors.

Let $(A, \Theta)$ be a principally polarized abelian variety (PPAV) of dimension $g$. Recall that by definition this means that $A = \CC^g / \Lambda$ is a $g$-dimensional complex torus, and $\Theta \subseteq A$ is an ample divisor with the property that 
\[
\hh{0}{A}{\OO_A(\Theta)} \ = \ 1. 
\]
The motivating example historically is the polarized Jacobian $(JC, \Theta_C)$ of a smooth projective curve of genus $g$.

In their classical work \cite{AM}, Andreotti and Meyer showed that Jacobians are generically characterized among all PPAV's by the condition that $\dim \text{Sing} (\Theta) \ge g - 4$.\footnote{The precise statement is that the Jacobians form an irreducible component of the locus of all $(A, \Theta)$ defined by the stated condition.} In view of this, it is interesting to ask what singularities can occur on theta divisors. Koll\'ar used vanishing for $\QQ$-divisors to prove a very clean statement along these lines.

Koll\'ar's result is the following:
\begin{theorem}
The pair $(A, \Theta)$ is log-canonical. In particular,
\[ \mult_x(\Theta) \ \le \ g \] for every $x \in A$.
\end{theorem}

\begin{proof}
Suppose to the contrary that $\MI{(1-\eps)\Theta} \ne \OO_A$ for some $\eps > 0$. We will derive a contradiction from Nadel vanishing. To this end, let $Z \subseteq A$ denote the subscheme defined by $\MI{(1-\eps)\Theta}$. Then $Z \subseteq \Theta$: this is clear set-theoretically, but in fact it holds on the level of schemes thanks to the inclusion
\[
\OO_A(-\Theta) \ = \ \MI{  \Theta} \ \subseteq \ \MI{(1-\eps)\Theta}. 
\]
Now consider the short exact sequence
\[
0 \lra \OO_A(\Theta) \otimes \MI{(1-\eps)\Theta} \lra \ \OO_A(\Theta) \lra \OO_Z(\Theta) \lra 0
\]
The $H^1$ of the term on the left vanishes thanks to Theorem \ref{Nadel.Vanishing.Theorem} and the fact that $K_A = 0$. Therefore the map
\[
\HH{0}{A}{\OO_A(\Theta) } \lra \HH{0}{Z}{\OO_Z(\Theta)}
\]
is surjective. On the other hand, the unique section of $\OO_A(\Theta)$ vanishes on $Z$, and so we conclude that
\begin{equation}
\HH{0}{Z}{\OO_Z(\Theta)} \ = \ 0.
\tag{*}
\end{equation}
To complete the proof, it remains only to show that (*) cannot hold. To this end, let $a \in A$ be a general point. Then $\Theta + a$ meets $Z$ properly, and hence $\HH{0}{Z}{\OO_Z(\Theta +a)} \ne 0$. Letting $a \to 0$, if follows by semicontinuity that
\[ \HH{0}{Z}{\OO_Z(\Theta)} \ \ne \ 0,
\] as required. 
\end{proof}

\begin{remark}
In the situation of the theorem, the fact that $(A, \Theta)$ is log-canonical implies more generally that  the locus 
\[
\Sigma_k(\Theta) \ =_{\text{def}} \ \big \{ x \in A\,  \big    | \, \mult_x \Theta \ge k \big \}
\]
 of $k$-fold points of $\Theta$ has codimension $\ge k$ in $A$. Equality is achieved when 
 \[ (A, \Theta) \ = \ (A_1, \Theta_1) \, \times \, \ldots \, \times \, (A_k , \Theta_k) \]
is the product of $k$ smaller PPAV's. It was established in \cite{STIV} that this is the only situation in which $\codim_A \Sigma_k(\Theta) = k$. It was also shown in that paper that if $\Theta$ is irreducible, then $\Theta$ is normal with rational singularities. \qed
\end{remark}

\subsection*{Uniform Global Generation}

As a final application, we prove a useful result to the effect that sheaves of the form $\OO(L) \otimes \MI{D}$, where $D$ is a $\QQ$-divisor numerically equivalent to $L$, become globally generated after twisting by a fixed divisor. This was first observed by Esnault and Viewheg, and later rediscovered independently by Siu. The statement plays an important role in the extension theorems of Siu discussed in Lecture 5.

The theorem for which we are aiming is the following:
\begin{theorem} \label{Unif.Gl.Gen.Thm}
Let $X$ be a smooth projective variety of dimension $d$. There exists a divisor $B$ on $X$ with the following property:
\begin{itemize}
\item For any divisor $L$ on $X$; and
\item For any effective $\QQ$-divisor $D \num L$,
\end{itemize}
the sheaf 
$\OO_X(L + B) \otimes \MI{D}$is globally generated.
\end{theorem}
\noi Note that the hypothesis implies that $L$ is $\QQ$-effective, i.e. that $\HH{0}{X}{\OO_X(mL)} \ne 0$ for some $m \gg 0$.  The crucial point  is that $B$ is independent of the choice of $L$ and $D$. 

\begin{corollary}
There is a fixed divisor $B$ on any smooth variety $X$ with the property that
\[ \HH{0}{X}{\OO_X(L+B)} \ \ne \ 0 \]
for any big \tn{(}or even $\QQ$-effective\tn{)} divisor $L$ on $X$. \qed\end{corollary}

The Theorem is actually an immediate consequence of Nadel vanishing and the elementary  lemma of Castelnuovo--Mumford:
\begin{lemma}[Castelnouvo--Mumford] Let $\FF$ be a coherent sheaf on a projective variety $X$, and let $H$ be a basepoint-free ample divisor on $X$. Assume that
\[  \HH{i}{X}{\FF \otimes \OO_X(-iH)} \ = \ 0 \ \ \text{for } \ i > 0. \]
Then $\FF$ is globally generated.
\end{lemma}
\noi We refer to \cite[\S 1.8]{PAG} for the proof.

\begin{proof} [Proof of Theorem \ref{Unif.Gl.Gen.Thm}] As above, let $ d = \dim X$. 
It sufficies to take $B = K_X + (d +1)H$ for a very ample divisor $H$, in which case Theorem \ref{Nadel.Vanishing.Theorem} gives the vanishings required for the Castelnuovo--Mumford lemma.
\end{proof}


\section{Local Properties of Multiplier Ideals}

In this Lecture we will discuss some local properties of multiplier ideals. First we take up the restriction theorem:  here we emphasize the use of adjoint ideals, whose importance has lately come into focus. The remaining sections deal with the subadditivity and Skoda theorems. As an application  of the latter, we give a down-to-earth discussion of the recent  results of \cite{LazLee} concerning syzygetic properties of multiplier ideals. 

\subsection*{Adjoint Ideals and the Restriction Theorem}

Let $X$ be a smooth complex variety, let $D$ be an effective $\QQ$-divisor on $X$, and let $S \subseteq X$ be a smooth irreducible divisor, not contained in any component of $D$. Thus the restriction $D_S$ of $D$ to $S$ is a well-defined $\QQ$-divisor on $S$.

There are now two multiplier-type ideals one can form on $S$. First, one can take the multiplier ideal $\MI{X,D}$ of $D$ on $X$, and then restrict this ideal to $S$. On the other hand, one can form the multiplier ideal $\MI{S, D_S}$ on $S$ of the restricted divisor $D_S$. In general, these two ideals are different:
\begin{example}
Let $X = \CC^2$, let $S$ be the $x$-axis, and let $A = \{y - x^2 = 0\}$ be a parabola tangent to $S$. If $D = \tfrac{1}{2}A$, then
\[
\MI{X, D} \ = \ \OO_X \ \ , \ \ \MI{S, D_S} \ = \ \OO_S(-P),
\]
where $P \in S$ denotes the origin. \qed \end{example}

However a very basic fact is that  there is a containment between these two ideal sheaves on $S$. \begin{theorem}[Restriction Theorem] 
\label{Restriction.Theorem}
One has an inclusion
\[ \MI{S, D_S} \ \subseteq \ \MI{X,D} \cdot \OO_S. \]
\end{theorem}
\noi This result is  perhaps the most important local property of multiplier ideals. 
In the analytic perspective, it comes from the Osahawa--Takegoshi extension theorem: an element in the ideal on the left is a function on $S$ satisfying an integrability condition, and Osahawa--Takegoshi guarantees that it is the restriction of a function satisfying the analogous integrability condition on $X$. 

We will prove Theorem \ref{Restriction.Theorem}
 by constructing and studying the \textit{adjoint ideal} $\Adj{S}{X, D}$ of $D$ along $S$. This is an ideal sheaf on $X$ that governs the multiplier ideal $\MI{S, D_S}$ of the restriction $D_S$ of $D$ to $S$.
\begin{theorem}\label{Adj.Ideal.Thm}
With hypotheses as above, there exists an ideal sheaf
\[ \Adj{S}{X,D}\ \subseteq \ \OO_X\] sitting in an exact sequence:
\begin{equation} 
\label{Adj.Exact.Seq}
0 \lra \MI{X,D} \otimes \OO_X(-S) \overset{\cdot S}  \lra \Adj{S}{X,D} \lra \MI{S, D_S} \lra 0. \end{equation}
Moreover, for any $0 < \eps \le 1$\tn{:}
\begin{equation} \label{Adj.Ideal.Inclusion}
\Adj{S}{X,D}\ \subseteq \ \MI{X, D + (1-\eps)S}. \end{equation}
\end{theorem}
 
The sequence \eqref{Adj.Exact.Seq} shows that
\[  \Adj{S}{X,D} \cdot \OO_S \ = \ \MI{S, D_S}. \]
Therefore \eqref{Adj.Ideal.Inclusion} not only yields the Restriction Theorem, it implies that in fact
\[
\MI{S, D_S} \ \subseteq \ \MI{X, D + (1-\eps)S}\cdot \OO_S
\]
for any $0 < \eps \le 1$.

Before proving Theorems \ref{Restriction.Theorem} and \ref {Adj.Ideal.Thm},  we record some consequences.

\begin{corollary} \label{Higher.Codi.Restr.Cor}
Let $Y \subseteq X$ be a smooth subvariety not contained in the support of $D$, so that the restriction $D_Y$ of $D$ to $Y$ is defined. Then
\[  \MI{Y, D_Y} \ \subseteq \ \MI{X,D} \cdot \OO_Y. \]
\end{corollary}
\noi (This follows inductively from the restriction theorem since $Y$ is locally a complete intersection in $X$.) \qed

\begin{corollary}
In the situation of the Theorem, assume that $\MI{S,D_S}$ is trivial at a point $x \in S$. Then $\MI{X, D+ (1 - \eps)S}$ \tn{(}and hence also $\MI{X,D}$\tn{)} are trivial at $x$. \qed
\end{corollary}

\begin{exercise}
If $D$ is an effective $\QQ$-divisor on $X$ such that $\mult_x(D) < 1$, then $\MI{X,D}$ is trivial at $x$. (Using the previous corollary, take hyperplane sections to reduce to the case $\dim X = 1$, where the result is clear.) \qed
\end{exercise}

\begin{proof}[Proof of Theorem \ref{Adj.Ideal.Thm}] Let $\mu : X^\pr \lra X$ be a log resolution of $(X, D+S)$, and denote by $ S^\pr \subseteq X^\pr$ the proper transform of $S$, so that in particular $\mu_S : S^\pr \lra S$ is a log resolution of $(S, D_S)$. Write
\[  \mu^* S \ = \ S^\pr + R, \]
and put $B   =   K_{X^\pr / X}- [ \mu^* D]  - R$.
We define:
\[ \Adj{S}{X,D} \ = \ \mu_* \OO_{X^\pr} \big( B). \]
To establish the exact sequence \eqref{Adj.Exact.Seq}, note first that
\[   K_S \ \lin \ \big( K_X + S\big)_{|S}  \ \ \ , \ \ \  K_{S^\pr} \ \lin \ \big( K_{X^\pr} + S^\pr \big)_{|S^\pr} ,
\]
and hence\[   K_{S^\pr/S} \ = \ \big( K_{X^\pr/X} - R \big)_{|S^\pr}. \]
(One can check that this holds on the level of divisors, and not only for linear equivalence classes.) On $X^\pr$, where the relevant divisors have SNC support, rounding commutes with restriction. Therefore
\begin{align*}
\MI{S,D_S} \ &= \ \mu_{S,*} \OO_{S^\pr}(K_{S^\pr/S} - [\mu_S^* D_S]) \\
&= \ \mu_*\OO_{S^\pr}( B_{S^\pr} ).
\end{align*}
Observing that 
\[
B - S^\pr \ = \ K_{X^\pr/X} - [\mu^*D] - \mu^*S,\]
the  adjoint exact sequence \eqref{Adj.Exact.Seq} follows by pushing forward
\[
0 \lra \OO_{X^\pr}(B - S^\pr) \lra \ \OO_{X^\pr}(B) \lra \OO_{S^\pr}(B_{S^\pr}) \lra 0
\]
  since the higher direct images of the term on the left vanish thanks to local vanishing. Finally, note that
  \begin{align*}
  B \ = \ K_{X^\pr/X} - [\mu^* D] - R \ &\preccurlyeq  \ K_{X^\pr/X} -\big[\mu^* D + (1-\eps) R\big]  \\ &= \ K_{X^\pr/X} -\big[\mu^* D + (1-\eps) R+ (1 - \eps)S^\pr\big],
  \end{align*}
  which yields \eqref{Adj.Ideal.Inclusion}. 
\end{proof}

\begin{remark}
We leave it to the reader to show that $\Adj{S}{X,D}$ is independent of the choice of log resolution.  \qed
\end{remark}

 If $\fra \subseteq \OO_X$ is an ideal that does not vanish identically on $S$, then for $c> 0$ one can define in the analogous manner an adjoint ideal \[ \Adj{S}{X, \fra^c} \ \subseteq \ \OO_X\] sitting in exact sequence
\[
0 \lra \MI{X,\fra^c}\otimes \OO_X(-S) \lra \Adj{S}{X, \fra^c} \lra \MI{S, (\fra _S)^c} \lra 0,
\]
where $\fra_S =_{\text{def}} \fra \cdot \OO_S$ is the restriction of $\fra$ to $S$. 
In particular, the analogue of the restriction theorem holds for the multiplier ideals associated to $\fra$.

Finally, Theorem \ref{Adj.Ideal.Thm} works  perfectly well if $S$ is allowed to be singular, as in Remark \ref{Mult.Ideal.Sing.Var.Remark}. Since in any event $S$ is Gorenstein, when in addition it is normal there is no question about the meaning of the multiplier ideals on $S$ appearing in the Theorem.  In this case the statement and proof of  \ref{Adj.Ideal.Thm} remain valid without change. For arbitrary $S$ one can twist by $\OO_X(K_X + S)$ and rewrite \eqref{Adj.Exact.Seq} as
\begin{multline} \label{SIngular.Adjoint.Sequence}
0 \lra \OO_X(K_X) \otimes \MI{X, D} \lra \OO_X(K_X + S) \otimes \Adj{S}{X,D} \\\lra  \mu_* \big( \OO_{S^\pr}(K_{S^\pr}) \otimes \MI{S^\pr , D_S^\pr} \big) \lra 0,
\end{multline}
where $D_{S^\pr} = \mu^*D_S$ denotes the pullback of $D_S$ to $S^\pr$.  
When $D=0$ this is the adjoint exact sequence appearing in \cite{STIV} and \cite[Section 9.3.E]{PAG}.

We conclude with some exercises for the reader.

\begin{exercise}[Irreducible plane curves] \label{Postulation.Irred.Curves}
Let $C \subseteq \PP^2$ be an irreducible (reduced) plane curve of degree $m$, and as in Proposition  \ref{Sings.of.Curves} put $\Sigma = \text{Sing}(C)$. Then the points of $\Sigma$ impose independent conditions on curves of degree $\ge m - 3$. (Let $f : C^\pr \lra C$ be the desingularization of $C$, and use  the adjoint sequence
\[
0 \lra \OO_{\PP^2}(-3) \lra \OO_{\PP^2}(m-3)\otimes \tn{Adj}_C  \lra f_* \OO_{C^\pr}(K_{C^\pr}) \lra 0
\]
coming from \eqref{SIngular.Adjoint.Sequence} with $D = 0$ to show that
\[
\HH{1}{\PP^2}{ \OO_{\PP^2}(k) \otimes \tn{Adj}_C} \ = \ 0
\]
when $k \ge m-3$.) \qed \end{exercise}

\begin{example}[Condition for an embedded point] It is sometimes interesting to know that a multiplier ideal has an embedded prime ideal. As usual, let $X$ be a smooth variety of dimension $d$, and fix a point $x \in X$ with maximal ideal $\frakm$. Consider an effective $\QQ$-divisor $D$ with integral multiplicity $s = \mult_x(D) \ge d$, and denote by $\overline D$ the effective $\QQ$-divisor of degree $s$ on $\PP(T_xX) = \PP^{d-1}$ arising in the natural way as the ``projectivized tangent cone" of $D$ at $x$. The following proposition asserts that if there is a hypersurface of degree $s - d$ on $\PP^{d-1}$ vanishing along the multiplier ideal $\MI{\PP^{d-1}, \overline{D}}$, then $\MI{X, D}$  has an embedded point at  $x$:
\begin{propositionn}
If \[
\HHH{0}{\PP^{d-1}}{\MI{\PP^{d-1}, \overline{D} }\otimes \OO_{\PP^{d-1}}(s - d)}\ \ne \ 0, \]
then $\frakm$ is an associated prime of $\MI{X,D}$.\end{propositionn}
\noi Observe that the statement is interesting only when $\MI{X,D}$ is not itself cosupported at $x$.

For the proof, let $\pi : X^\pr \lra X$ be the blowing up of $x$, with exceptional divisor $E = \PP^{d-1}$, so that $K_{X^\pr/X} = (d-1)E$.  Write $D^\pr$ for the proper transform of $D$, so that $D^\pr \num \pi^* D - sE$, and note that $D_E = \overline D$. Now consider the adjoint sequence for $D^\pr$:
\[
0 \lra \MI{X^\pr, D^\pr}\otimes \OO_{X^\pr}(-E) \lra \Adj{E}{X^\pr, D^\pr}\lra \MI{\PP^{d-1}, \overline{D}}\lra 0, \]
and twist through by $\OO_{X^\pr}\big((d-s)E\big)$. The resulting term on the left pushes down with vanishing higher direct images to $\MI{X,D}$ thanks to Lemma \ref {Birat.Transf.Rule}. One finds an exact sequence having the shape
\[
0 \lra \MI{X, D} \lra \mathcal{A} \lra  \HHH{0}{E}{\MI{E, \overline D} \otimes \OO_E((d-s)E)} \lra 0,\]
where $\mathcal{A}$ is an ideal on $X$, and the vector space on the right is viewed as a sky-scraper sheaf supported at $x$. But the hypothesis of the Proposition is exactly that this vector space is non-zero, and the assertion follows. \qed
\end{example}

\subsection*{The Subadditivity Theorem}

The Restriction Theorem was applied in \cite{DEL} to prove a result asserting that the multiplier ideal of a product of two ideals must be at least as deep as the product of the corresponding multiplier ideals. This will be useful at a couple of points in the sequel. 

We start by defining ``mixed" multiplier ideals:
\begin{definition} \label{Def.Mixed.Mult.Ideals}
Let $\fra, \frb\subset \OO_X$ be two ideal sheaves. Given $c, e \ge 0$, the multiplier ideal
\[  \MI{\fra^c \cdot \frb^e} \ \subseteq \ \OO_X \]
is defined by taking a common log resolution $\mu : X^\pr \lra X$ of $\fra$ and $\frb$, with
\[
\fra \cdot \OO_{X^\pr}  \ = \ \OO_{X^\pr}(-A) \ \ , \ \ \frb \cdot \OO_{X^\pr}  \ = \ \OO_{X^\pr}(-B)
\]for divisors $A, B$ with SNC support, and setting
\[ 
\MI{\fra^c\cdot \frb^e} \ = \ \mu_* \OO_{X^\pr}\big(K_{X^\pr/X} - [ cA + eB]\big). \ \  \qed
\]
\end{definition}
The subadditivity theorem compares these mixed ideals to the multiplier ideals of the two factors.
\begin{theorem}[Subadditivity Theorem]
\label{Subadditivity.Theorem}
One has an inclusion
\[ 
\MI{\fra^c \cdot \frb^e} \ \subseteq \ \MI{\fra^c} \cdot \MI{\frb^e}.
\] 
\end{theorem}
\noi Similarly,
$\MI{X, D_1 + D_2} \subseteq \MI{X, D_1} \cdot \MI{X, D_2} $
for any two effective $\QQ$-divisors $D_1, D_2$ on $X$

\begin{proof}[Sketch of Proof of Theorem \ref{Subadditivity.Theorem}]
Consider the product $X \times X$ with projections
\[ p_1\, , \, p_2 \, : \, X \times X \lra X. \]
The first step is to show via the K\"unneth formula that
\begin{equation}
\MI{X \times X, (p_1^* \fra)^c\cdot (p_2^* \frb)^e} \ = \ p_1^* \MI{X, \fra^c} \ \cdot \ p_2^* \MI{X, \frb^e}.
\tag{*}
\end{equation}
The one simply restricts to the diagonal $\Delta = X \subseteq X \times X$ using Corollary \ref{Higher.Codi.Restr.Cor}. Specifically, 
\begin{align*}
\MI{X , \fra^c \cdot \frb^e } \ &= \ \MI{\Delta,((p_1^* \fra)^c\cdot (p_2^* \frb)^e)_{|\Delta}} \\
&\subseteq \  \MI{X \times X, (p_1^* \fra)^c\cdot (p_2^* \frb)^e}_{| \Delta} \\
& = \  \MI{X, \fra^c} \cdot \MI{X, \frb^e},
\end{align*}
the last equality coming from (*).
\end{proof}

\begin{exercise}
Let $f : Y \lra X$ be a morphism, and let $D$ be a $\QQ$-divisor on $X$ whose support does not contain the image of $Y$. Then
\[  \MI{Y, f^*D}  \ \subseteq \ f^{-1} \MI{X, D}. \ \qed \]
\end{exercise}

\subsection*{Skoda's Theorem}

We now discuss Skoda's theorem, which computes the multiplier ideals associated to powers of an ideal. 

Let $X$ be a smooth variety of dimension $d$, and let $\fra, \frb \subseteq \OO_X$ be ideal sheaves. In its simplest form, Skoda's theorem is the following:
\begin{theorem} [Skoda's Theorem] 
\label{Skoda.Thm.Statement}
Assume that  $m \ge d$. Then
\[  \MI{\fra^m \cdot \frb^c} \ = \ \fra \cdot \MI{\fra^{m-1} \cdot \frb^c}\]
for any $c \ge 0$. 
\end{theorem}
\noi Note that it follows that $\MI{\fra^m \frb^c} = \fra^{m+1-d}\MI{\fra^{d-1} \frb^c}$.  

The algebraic proof of the theorem -- as in \cite{ELNull} or \cite[\S 10.6]{PAG} -- actually yields a more general statement that in turn has some interesting consequences. 
Before stating this, we record an elementary observation, whose proof we leave to the reader:
\begin{lemma} \label{Inclusion.Lemma}
For any $\ell, k \ge 0$ there is an inclusion
\[
\fra^k \cdot \MI{\fra^\ell \cdot \frb^c} \ \subseteq \ \MI{\fra^{k + \ell} \cdot \frb^c}. \qed
\]
\end{lemma}

Theorem \ref{Skoda.Thm.Statement}
 will follow from the exactness of certain ``Skoda complexes." Specifically, assume that $X$ is affine, fix any point $x \in X$, and choose $d$ general elements
 \[  f_1, \ldots, f_d \ \in \ \fra.\] 
\begin{theorem} \label{Skoda.Complex.Exact}
Still supposing that $m \ge d = \dim X$, the $f_i$  determine a Koszul-type complex
\begin{multline*}
0 \lra \Lambda^d \OO^d \otimes \MI{\fra^{m-d}\frb^c} \lra \Lambda^{d-1} \OO^d \otimes \MI{\fra^{m+1-d} \frb^c}  \lra \ldots \\ 
\ldots \lra
\Lambda^2  \OO^d \otimes \MI{\fra^{m-2} \frb^c} \lra \OO^d \otimes \MI{\fra^{m-1} \frb^c}  \lra \MI{\fra^{m} \frb^c} \lra 0
\end{multline*}
that is exact in a neighborhood of $x$. 
\end{theorem}
\noi The homomorphism $ \OO^d \otimes \MI{\fra^{d-1} \frb^c}  \lra \MI{\fra^{d} \frb^c}$ on the right is given by multiplication by the vector $(f_1, \ldots, f_d)$. The surjectivity of this map implies that
\[
\MI{ \fra^m \frb^c} \ = \ (f_1, \ldots, f_d) \cdot \MI{\fra^{m-1} \frb^c}.
\]
But thanks to the Lemma one has
\[
(f_1, \ldots, f_d) \cdot \MI{\fra^{m-1} \frb^c} \ \subseteq \ \fra \cdot \MI{\fra^{m-1} \frb^c} \ \subseteq \ \MI{\fra^m \frb^c}, \]
so Skoda's theorem follows.

\begin{proof} [Proof of Theorem \ref{Skoda.Complex.Exact}]
Let $\mu : X^\pr \lra X$ be a log resolution of $\fra$ and $\frb$ as in Definition \ref{Def.Mixed.Mult.Ideals}. We keep the notation of that definition, so that in particular $\fra \cdot \OO_{X^\pr} = \OO_{X^\pr}(-A)$. The $d$ general elements $f_1, \ldots. f_d \in \fra$ determine sections 
\[ f_i^\pr \ \in \ \HH{0}{X^\pr}{\OO_{X^\pr}(-A)}, \]
and an elementary dimension count shows that after possibly shrinking $X$ one can suppose that these sections actually generate $\OO_{X^\pr}(-A)$ (cf. \cite[9.6.19]{PAG}). The $f_i^\pr$ then determine an exact Koszul complex
\begin{multline} \label{Koszul.Cx.For.Skoda}
0 \lra \Lambda^d \OO^d \otimes \OO_{X^\pr}(-(m-d)A)  \lra \Lambda^{d-1} \OO^d \otimes\OO_{X^\pr}(-(m-d-1)A)  \lra \ldots \\ 
\ldots \lra
\Lambda^2  \OO^d \otimes \OO_{X^\pr}(-(m-2)A) \lra \OO^d \otimes \OO_{X^\pr}(-(m-1)A) \lra \OO_{X^\pr}(-mA)\lra 0
\end{multline}
of vector bundles on $X^\pr$. 
Now twist \eqref{Koszul.Cx.For.Skoda} through by $\OO_{X^\pr}\big(K_{X^\pr/X} - [cB]\big)$. The higher direct images of all the terms of the resulting complex vanish thanks to Theorem \ref{Local.Van.Mult.Ideals}. This implies the exactness of the complex on $X$ obtained by taking direct images of the indicated twist of  \eqref{Koszul.Cx.For.Skoda}, which is exactly the assertion of the Theorem. 
\end{proof}

\begin{remark}
The functions $f_1, \ldots, f_r \in \fra$ occuring in Theorem \ref{Skoda.Complex.Exact} do not necessarily generate $\fra$. Rather (after shrinking $X$) they generate an ideal $\mathfrak{r} \subseteq \fra$ with the property that 
\[
\mathfrak{r} \cdot \OO_{X^\pr} \ = \ \fra \cdot \OO_{X^\pr},\]
which is equivalent to saying the $\frr$ and $\fra$ have the same integral closure. Such an ideal $\frr$ is called a ``reduction" of $\fra$. See \cite[\S 10.6.A]{PAG} for more details. \qed
\end{remark}

Following \cite{LazLee}, we use Theorem \ref{Skoda.Complex.Exact}
 to show that multiplier ideals satisfy some unexpected syzygetic conditions. By way of background, it is natural to ask which ideals $\frd \subseteq \OO_X$ can be realized as a multiplier ideal 
$\frd = \MI{\frb^c}$ for some $\frb$ and $c \ge 0$. It follows from the definition that multiplier ideals are integrally closed, meaning that membership in a multiplier ideal is tested by order of vanishing along some divisors over $X$. However until recently, multiplier ideals were not known to satisfy any other local properties. In fact, Favre--Jonsson \cite{FJ} and Lipman--Watanabe \cite{LW} showed that in dimension $d = 2$, any integrally closed ideal is locally a multiplier ideal. 

The next theorem implies that the corresponding statement is far from true in dimensions $d \ge 3$. We  work in the local ring $(\OO, \frakm)$ of $X$ at a point $x \in X$. 
\begin{theorem} \label{LSMI.Theorem}
Let \
\[  \frj \ = \  \MI{\frb^c}_x \ \subseteq \ \OO \] be the germ at $x$ of some multiplier ideal, and choose minimal generators
\[  h_1, \ldots, h_r \ \in \ \frj \]
of $\frj$. Let $b_1, \ldots, b_r \in \frakm$ be functions giving a minimal syzygy 
\[
\sum b_i h_i \ = \ 0
\]
among the $h_i$. Then there is at least one index $i$ such that
\[
\ord_x(b_i) \ \le \ d-1. 
\]
\end{theorem}
\noi To say that the $h_i$ are minimal generators means by definition that they determine a basis of the $\OO/\frakm = \CC$-vector space $\frj / \frakm \cdot \frj$,
the hypothesis on the $b_i$ being similar. Note that there are no restrictions on the order of vanishing of \textit{generators} of a multiplier ideal, since for instance $\frakm^\ell = \MI{\frakm^{\ell + d - 1}}$ for any $\ell \ge 1$. On the other hand, the Theorem extends to  statements for the higher syzygies of $\frj$, for which we refer to \cite{LazLee}.

\begin{example}
Assume that $d = \dim X \ge 3$, choose two general functions $h_1, h_2 \in\frakm^p$ vanishing to order $p \ge d$ at $x$, and consider the complete intersection ideal \[ \frd \ =\  (h_1, h_2)\ \subseteq \ \OO. \]
Since $d \ge 3$, the zeroes of $\frd$ will be a reduced algebraic set of codimension $2$, and therefore $\frd$ is integrally closed. On the other hand, the only minimal syzygy among the $h_i$ is the Koszul relation
\[  (h_2) \cdot  h_1 \, + \, (-h_1) \cdot h_2 \ = \ 0. \]
In particular, the conclusion of the Theorem does not hold, and so $\frd$ is not a multiplier ideal. (When $d =2$, the ideal $\frd$ is not integrally closed.)  \qed
\end{example}

\begin{proof}[Idea of Proof of Theorem \ref{LSMI.Theorem}] The plan is to apply Theorem \ref{Skoda.Complex.Exact} with $\fra = \frakm$. Specifically, assume for a contradiction that each of the $b_i$ vanishes to order $\ge d$, and choose local coordinates  $z_1, \ldots , z_d$ at $x$, so that $\frakm = (z_1, \ldots, z_s)$. We can write
\[
b_1 \ = \ z_1 b_{11} + \ldots  + z_d b_{1d} \ , \ \ \cdots \ \ , \ b_r \ = \ z_1 b_{r1} + \ldots + z_d b_{rd},
\]
for some functions $b_{ij}$ vanishing to order $\ge d-1$ at $x$. Now put
\[
G_1 \ = \ b_{11} h_1 + \ldots + b_{r1}h_r \ , \ \ \cdots \ \ , \ G_d = b_{1d}h_1 + \ldots + b_{rd}h_d. 
\]
Then $G_j \in \MI{\frakm^{d-1}\frb^c}$ thanks to Lemma \ref{Inclusion.Lemma}, and
\begin{equation}
z_1 G_1 + \ldots + z_d G_d \ = \ 0 \tag{*}
\end{equation}
 by construction. The relation  (*) means that $(G_1, \ldots, G_d)$ is a cycle for the Skoda complex
 \begin{equation}
 \OO^{\binom{n}{2}} \otimes \MI{\frakm^{d-2}\frb^c} \lra \OO^d \otimes \MI{\frakm^{d-1}\frb^c} \lra \MI{\frakm^d \frb^c}  \tag{**}
 \end{equation}
 and using the fact that the $b_i$ are minimal one can show via some Koszul cohomology arguments that  it gives rise to a non-trivial cohomology class in (**). But this contradicts the exactness of (**).
 \end{proof}

\begin{example} [Skoda's theorem and the effective Nullstellensatz]
\label{Skoda.Nullstellensatz}
One can combine Skoda's theorem with jumping numbers to give statements in the direction of the effective Nullstellensatz. Specifically, let  $\fra \subseteq \OO_X$ be an ideal. Then there is an integer $s > 0$ such that \[(\sqrt{\fra})^s \ \subseteq \ \fra, \] and it is interesting to ask for effective bounds for $s$: see for instance \cite{Kollar.Null} and \cite{ELNull}, or \cite[\S 10.5]{PAG} for a survey and references.  Now fix a point point $x \in X$, and consider the jumping numbers $\xi_i = \xi_i(\fra;x)$ of $\fra$ at $x$ (Proposition/Definition \ref{Propdef.Jumping.Nos}). Let $\sigma = \sigma(\fra;x)$ be the least index such that 
$\xi_\sigma \ge d$.  Then
\[  \MI{\fra^{\xi_{\sigma}} }_x \ \subseteq \ \MI{\fra^d}_x \ \subseteq \ (\fra)_x , \]
thanks to Skoda's theorem, so it follows from Execise \ref{Mustata.Observation}
that 
\[ ( \sqrt{\fra})^{\sigma(\fra;x)} \ \subseteq \ \fra
\]
in a neighborhood of $x$. An analogous global statement holds using non-localized jumping numbers. It would be interesting to know whether one can recover or improve the results of \cite{Kollar.Null} or \cite{ELNull} by using global arguments to bound $\sigma$. (The arguments in \cite{ELNull} also revolve around Skoda's theorem, but from a somewhat different perspective.)  \qed
\end{example}


\section{Asymptotic Constructions}

In this lecture we will study asymptotic constructions that can be made with multiplier ideals. It is important in many geometric problems to be able to analyze for example the linear systems $\linser{mL}$ associated to arbitrarily large multiples of a given divisor. Unfortunately one cannot in general find one birational model of $X$ on which these are all well-behaved. By contrast, it turns out that there is some finiteness built into multiplier ideals, and the constructions  discussed here  are designed to exploit this. 

\subsection*{Aymptotic Multiplier Ideals}
In this section we construct the asymptotic multiplier ideals associated to a big divisor. For the purposes of motivation, we start by defining their non-asymptotic parents, which we have not needed up to now. 

Let $X$ be a smooth projective variety, and $L$ a divisor on $X$ such that the complete linear series $\linser{L}$ is non-trivial.
Given $c > 0$ we construct a multiplier ideal
\[ \MI{c \cdot \linser{L}}\  \subseteq \ \OO_X\] as follows. Take $\mu : X^\pr \lra X$ to be a log resolution of $\linser{L}$: this means that $\mu$ is a projective birational morphism, with 
\[  \mu^* \linser{L} \ = \ \linser{M} \, + \, F, \]
where $\linser{M}$ is a basepoint-free linear series, and $F + \exc({\mu})$ has SNC support. 
(This is the same thing as a log-resolution of the base-ideal $\bs{L} \subseteq \OO_X$ of $\linser{L}$.)
One then defines
\[  \MI{c \cdot \linser{L} } \ = \ \mu_* \OO_{X^\pr}\big( K_{X^\pr/X} - [cF] \big). \]
One can think of these ideals as measuring the singularities of the general divisor $A \in \linser{L}$. 
 
The multiplier ideals attached to a linear series enjoy a Nadel-type vanishing:
\begin{theorem} \label{Van.For.MI.Lin.Ser}
Assume that $B$ is  a nef and big divisor on $X$. Then 
\[
\HH{i}{X}{\OO_X(K_X + L + B) \otimes \MI{\linser{L}}} \ = \ 0
\]
for $i > 0$. \end{theorem}
\noi A proof is sketched in the following Exercise.

\begin{exercise} \label{MI.Lin.Series.Ex.1}
Choose $k >c$ general divisors
$A_1, \ldots, A_k \in \linser{L}$, and let 
\[  D \ = \ \frac{1}{k} \cdot \big( A_1 + \ldots + A_k). \]
Then
\[
\MI{c \cdot \linser{L}} \ = \ \MI{cD}. \]
In particular, Theorem \ref{Van.For.MI.Lin.Ser} is a consequence of Theorem \ref{Nadel.Vanishing.Theorem}.  \qed
\end{exercise}

\begin{exercise} \label{MI.Lin.Series.Ex.2}Let $\frb = \bs{L} \subseteq \OO_X$ be the base-ideal of $\linser{L}$. Then
\[  \MI{c \cdot \linser{L} } \ = \ \MI{\frb^c}.  \ \qed \]
\end{exercise}

\begin{remark} [Incomplete linear series] Starting with a non-trivial linear series 
$\linser{V} \subseteq \linser{L}$, one constructs in the similar manner a multiplier ideal $\MI{c \cdot \linser{V}}$. The analogue of Theorem \ref{Van.For.MI.Lin.Ser} holds for these, as do the natural extensions of Exercises \ref{MI.Lin.Series.Ex.1} and \ref{MI.Lin.Series.Ex.2}. The reader may consult \cite[\S 9.2]{PAG} for details. \qed \end{remark}

\begin{exercise} (Adjoint ideals for linear series).
\label{Adj.Ideals.Lin.Series}
With $X$ as above, suppose that $S \subseteq X$ is a smooth irreducible divisor not contained in the base locus of $\linser{L}$. Then, as in the previous Lecture,  one constructs for $c > 0$ an adjoint ideal \[ \Adj{S}{X, c \cdot \linser{L}}\  \subseteq \  \OO_X. \]
This  sits in an exact sequence
\[
0 \lra \MI{X, c\cdot \linser{L} } \otimes \OO_X(-S) \lra \Adj{S}{X, c \cdot \linser{L} } \lra \MI{S, c \cdot \linser{L}_S } \lra 0,
\]
where the term $\linser{L}_S$ on the right involves the (possibly incomplete) linear series on $S$ obtained as the restriction of  $\linser{L}$ from $X$ to $S$. (One can use Exercise \ref{MI.Lin.Series.Ex.1} to reduce to the case of divisors.) \qed  \end{exercise}

The ideals $\MI{c \cdot \linser{L}}$ reflect the geometry of the linear series $\linser{L}$. But a variant of this construction gives rise to ideals that involve the asymptotic geometry of $\linser{mL}$ for all $m \gg 0$. 
\begin{propdef} \label{Prop.Def.Defining.AMI}
Assume that $L$ is big, and fix $c > 0$. Then for $p \gg 0$ the multiplier ideals
\[  \MI{X, \tfrac{c}{p} \cdot \linser{pL}} \]
all coincide. The resulting ideal, written $\MI{X, c \cdot \alinser{L}}$, is the \tn{asymptotic multiplier ideal} of $\linser{L}$ with coefficient $c$.  
\end{propdef}

\begin{proof}[Idea of Proof] It follows from the Noetherian property that as $p$ varies over all positive integers, the family of ideals $\MI{\tfrac{c}{p}\cdot \linser{pL}}$ has a maximal element. But one checks  that 
\[
\MI{ \tfrac{c}{p}\cdot \linser{pL}}  \ \subseteq \ \MI{\tfrac{c}{pq}\cdot\linser{pqL}} \]
for all $q \ge 1$, and therefore the family in question has a \textit{unique} maximal element $\MI{c \cdot \alinser{L}}$, which coincides with $\MI{\tfrac{c}{p}\cdot \linser{pL}}$ for all sufficiently large and divisible $p$. For the fact that these agree with $\MI{\tfrac{c}{p}\cdot\linser{pL}}$ for any sufficiently large $p$ (depending on $c$), see \cite[Proposition 11.1.18]{PAG}.
\end{proof}

The asymptotic multiplier ideals associated to a big divisor $L$ are the algebro-geometric analogues of the multiplier ideals associated to metrics of minimal singularities in the analytic theory. It is conjectured -- but not known -- that the two sorts of multiplier ideals actually coincide provided that $L$ is big. (See \cite{DEL}).

The following theorem summarizes the most important properties of asymptotic multiplier ideals.
\begin{theorem}[Properties of asymptotic multiplier ideals]
\label{Properties.AMI.Thm}
Assume that $L$ is a big divisor on the smooth projective variety $X$.
\begin{enumerate}
\item[(i).] Every section of $\OO_X(L)$ vanishes along $\MI{\alinser{L}}$, i.e. the natural inclusion
\[  \HH{0}{X}{\OO_X(L) \otimes \MI{\alinser{L}}} \lra \HH{0}{X}{\OO_X(L)} \]
is an isomorphism. Equivalently, $\bs{L} \subseteq \MI{X, \alinser{L}}$. 
\sbl
\item[(ii).] If $B$ is nef, then
\[
\HH{i}{X}{\OO_X(K_X + L + B) \otimes \MI{\alinser{L}}} \ = \ 0
\]
for $i > 0$.
\lbl
\item[(iii).] $\MI{\alinser{mL}} = \MI{m \cdot \alinser{L}}$ for every positive integer $m$. 
\lbl
\item[(iv).] \tn{(}Subadditivity.\tn{)}  $\MI{\alinser{mL}} \subseteq \ \MI{\alinser{L}}^m$ for every $m \ge 0$. 
\lbl
\item[(v).] If $f : Y \lra X$ is an \'etale covering, then 
\[
\MI{Y, \alinser{f^*L}} \ = \ f^* \MI{X, \alinser{L}}.
\]
\end{enumerate}
\end{theorem}
\noi Concerning the vanishing in (ii), note that it is not required that $B$ be big. In particular, one can take $B = 0$. 
\begin{exercise}
Give counterexamples to the analogues of properties (ii) -- (v) if one replaces the asymptotic ideals by the multiplier ideals $\MI{\linser{L}}$ attached to a single linear series. (For example, suppose that $L$ is ample, but that $\linser{L}$ is not free. Then $\MI{m \cdot \linser{L}} \ne \OO_X$ for $m \gg 0$, whereas $\MI{\linser{mL}} = \OO_X$ for large $m$.) \qed \end{exercise}

\begin{exercise}
In the situation of the Theorem, one has
\[ \MI{c \cdot \alinser{L}} \ \subseteq  \ \MI{d \cdot \alinser{L}}\]
whenever $c \ge d$. \qed \end{exercise}

\begin{proof}[Indications of Proof of Theorem \ref{Properties.AMI.Thm}]
We prove (ii) and (iii) in order to give the flavor. For (ii), choose $p \gg 0$ such that $\MI{\alinser{L}} = \MI{\frac{1}{p}\linser{L}}$, and let 
$\mu : X^\pr \lra X$ be a log-resolution of $\linser{pL}$, with
\[  \mu^* \linser{pL} \ = \ \linser{M_p} \, + \, F_p, \]
where $\linser{M_p}$ is free. We can suppose that $M_p$ is big (since $L$ is). Arguing as in the proof of Theorem \ref{Nadel.Vanishing.Theorem}, one has
\[
\HH{i}{X}{\OO_X(K_X + L+B) \otimes \MI{\alinser{L}}} \ = \ \HH{i}{X^\pr}{\OO_{X^\pr}(K_{X^\pr} + \mu^*(L + B) - [\tfrac{1}{p}F_p])}.
\]
But \[ \mu^*(L + B)-\tfrac{1}{p}F_p \ \num \ \mu^*(B) + \tfrac{1}{p}M_p\] is nef and big, so the required vanishing follows from the theorem of Kawamata and Viehweg. For (iii), the argument is purely formal. Specifically, for $p \gg 0$ one has thanks to Proposition/Definition \ref{Prop.Def.Defining.AMI}
\begin{align*}
\MI{\alinser{mL}} \ &= \ \MI{\tfrac{1}{p}\cdot \linser{pmL}} \\
&= \ \MI{\tfrac{m}{mp}\cdot \linser{pmL}} \\
&= \ \MI{m \cdot \alinser{L}},
\end{align*}
where in the last equality we are using $mp$ in place of $p$ for the large index computing the asymptotic multiplier ideal in question. 
For the remaining statements we refer to \cite[11.1, 11.2]{PAG}. 
\end{proof}

\begin{exercise} [Uniform global generation] 
\label{Asympt.Univ.Gl.Gen.Ex}
Theorem \ref{Unif.Gl.Gen.Thm} extends to the asymptotic setting. Specifically, there exists a divisor $B$ on $X$ with the property that
$\OO_X(L + B) \otimes \MI{\alinser{L}}$ is globally generated for any big divisor $L$ on $X$. \qed
\end{exercise}

\begin{exercise}[Characterization of nef and big divisors] Let $L$ be a big divisor on $X$. Then $L$ is nef if and only if
\begin{equation}
\MI{\alinser{mL}} \ = \ \OO_X \ \ \text{ for all } m > 0. \tag{*}
\end{equation}
(Suppose (*) holds. Then it follows from the previous exercise that $\OO_X(mL + B)$
is globally generated for all $m > 0$ and some fixed  $B$. This implies that $\big ( (mL +B) \cdot C \big) \ge 0$ for every effective curve $C$ and every $m > 0$, and hence that $(L \cdot C) \ge 0$.)  \qed
\end{exercise}

\subsection*{Variants}
We next discuss some variants of the construction studied in the previous section. To begin with, one can deal with possibly incomplete linear series. For this, recall that a \textit{graded linear series} $W\bull = \{ W_m\} $ associated to a big divisor $L$ consists of subspaces \[
W_m \ \subseteq \ \HH{0}{X}{\OO_X(mL)},  \]
with $W_m \ne 0$ for $m \gg 0$, satisfying the condition that 
\[ R(W\bull) \ =_{\text{def}} \ \oplus W_m\] be a graded subalgebra of the section ring $R(L) = \oplus \HH{0}{X}{\OO_X(mL)}$ of $L$. 
This last requirement is equivalent to asking that 
\[  W_\ell \cdot W_m \ \subseteq \ W_{\ell + m},\]
where the left hand side denotes the image of $W_\ell \otimes W_m$ under the natural map \[\HH{0}{X}{\OO_X(\ell L)} \otimes \HH{0}{X}{\OO_X(m  L)}\lra \HH{0}{X}{\OO_X((\ell + m)L)} .\] Then just as above one gets an asymptotic multiplier ideal by taking
\[ \MI{c \cdot \alinser{W\bull}} \ = \ \MI{\tfrac{c}{p} \cdot \linser{W_p}} \]
for $p \gg 0$. Analogues of Properties (i) -- (iv) from Theorem \ref{Properties.AMI.Thm} remain valid in this setting; we leave precise statements and proofs to the reader.

\begin{example} [Restricted linear series] \label{Restr.Series.AMI}
An important example of this construction occurs when $Y \subseteq X$ is a smooth subvariety not contained in the stable base locus $\BBB(L)$ of the big line bundle $L$ on $X$. In this case, one gets a graded linear series on $Y$ by taking 
\[
W_m \ = \ \tn{Im} \Big ( \HH{0}{X}{\OO_X(mL)} \lra \HH{0}{Y}{\OO_Y(mL_Y)} \Big). 
\]
We denote the corresponding multiplier ideals by $\MI{Y, c \cdot \alinser{L}_Y}$.
Note that there is an inclusion 
\[ \MI{Y, c \cdot \alinser{L}_Y} \ \subseteq \ \MI{Y, c \cdot \alinser{L_Y}},\]
but these can be quite different. For example if the restriction $L_Y$ is ample on $Y$, then the ideal on the right is trivial. On the other hand, if  for all $m \gg 0$   the linear series $\linser{mL}$ on $X$ have base-loci that meet $Y$, then the ideals on the left could be quite deep. \qed
\end{example}
 
 One can also extend the construction of adjoint ideals to the asymptotic setting.
 Assume that $S \subseteq X$ is a smooth irreducible divisor which is not contained in the stable base locus $\BBB(L)$ of a big divisor $L$. 
 Then one defines an asymptotic adjoint ideal $\Adj{S}{X, \alinser{L}}$ that fits into an exact sequence:
\begin{equation} \label{Asympt.Adj.Seq}
0 \lra \MI{X, \alinser{L}}\otimes \OO_X(-S) \lra \Adj{S}{X, \alinser{L}} \lra \MI{S, \alinser{L}_S} \lra 0.
\end{equation}
The ideal on the right is the asymptotic multiplier ideal associated to the restricted linear series of $L$ from $X$ to $S$, as in the previous example. This sequence will play a central role in our discussion of extension theorems in the next Lecture.

Finally, we discuss the asymptotic analogues of the ideals $\MI{\fra^c}$. A \textit{graded family of ideals} $\fra\bull = \{ \fra_m\}$ on a variety $X$ consists of ideal sheaves
$\fra_m \subseteq \OO_X$ satisfying the property that
\[
\fra_{\ell} \cdot \fra_m \ \subseteq \ \fra_{\ell + m}.
\]
We will also suppose that $\fra_0 = \OO_X$ and that $\fra_m \ne (0)$ for $m \gg 0$. Then one defines:
\[
\MI{\fra\bull^c} \ = \ \MI{c \cdot \fra\bull} \ =_{\text{def}} \ \MI{(\fra_p)^{c/p} }
\]
for $p \gg 0$. 

\begin{example}
The prototypical example of a graded system of ideals is the family of base-ideals \[
\frb_m \ = \ \bs{mL}\]
associated to multiples of a big divisor $L$ when $X$ is projective. In this case
\[
\MI{\frb\bull^c} \ = \ \MI{c \cdot \alinser{L}}. \ \qed
\]
\end{example}

\begin{example}[Symbolic powers]
\label{symbolic.power.example} A second important example involves the symbolic powers of a radical ideal $\frq = \II_Z $ defining a (reduced) subscheme $Z \subseteq X$. Assuming as usual that $X$ is smooth, define 
\[
\frq^{(m)} \ = \ \big \{ f \in \OO_X \ \big | \ \ord_x(f) \ge m \ \text{ for general } x \in Z \big \}.
\]
(If $X$ is reducible, we ask that the condition hold at a general point of each irreducible component of $X$.) Observe that by construction, membership in $\frq^{(m)}$ is tested at a general point of $Z$, i.e. this is a primary ideal. The inclusion 
\[  \frq^{(\ell)} \cdot \frq^{(m)} \ \subseteq \ \frq^{(\ell  + m)} \]
being evident, these form a graded family denoted $\frq_{(\bullet)}$. \qed
\end{example}

\begin{exercise} \label{Property.GSI}
Let $\fra\bull$ be a graded family of ideals. Then for every $m \ge 1$:
\begin{enumerate}
\item[(i).] $\fra_m \subseteq \MI{\fra\bull^m} $;
\lbl
\item[(ii).]  $\MI{\fra\bull^m} \subseteq \MI{\fra\bull}^m$. \ \qed \end{enumerate}
\end{exercise}

\subsection*{\'Etale Multiplicativity of Plurigenera}

As a first illustration of this machinery, we prove a theorem of Koll\'ar concerning the behavior of plurigenera under \'etale covers.

Given a smooth projective variety $X$, recall that the $m^{\text{th}}$ plurigenus $P_m(X)$ of $X$ is the dimension of the space of $m$-canonical forms on $X$:
\[  P_m(X)  \ =_{\text{def}} \ \hh{0}{X}{\OO_X(mK_X)}. \]
These plurigeneral are perhaps the most basic birational invariants of a variety.

Koll\'ar's theorem is that these behave well under \'etale coverings:
\begin{theorem}
\label{Kollar.Plurigenera}
 Assume that $X$ is of general type, and 
let $f : Y \lra X$ be an unramified covering of degree $d$. Then for $m \ge 2$,
\[  P_m(Y) \ = \ d \cdot P_m(X).\]
\end{theorem}
\noi Koll\'ar was led to this statement by the observation that it would (and now does) follow from the minimal model program.

\begin{exercise} Show that the analogous statement can fail when $m = 1$. (Consider the case $\dim X = 1$.) \qed 
\end{exercise}

\begin{proof}[Proof of Theorem \ref{Kollar.Plurigenera}]
We use the various properties of asymptotic multiplier ideals given in Theorem \ref{Properties.AMI.Thm}. To begin with, \ref{Properties.AMI.Thm} (i) yields:\begin{align*}
\HH{0}{X}{\OO_X(mK_X)} \ &= \ \HH{0}{X}{\OO_X(mK_X) \otimes \MI{\alinser{mK_X}} }\\
&= \ \HH{0}{X}{\OO_X(mK_X) \otimes \MI{\alinser{(m-1)K_X}}},
\end{align*}
the second equality coming from the includion $\MI{\alinser{(m-1)K_X}} \subseteq \MI{\alinser{mK_X}}$.
But since $X$ is of general type, when $m \ge 2$ the vanishing statement (ii) implies that
\[
\HH{i}{X}{\OO_X(mK_X) \otimes \MI{\alinser{(m-1)K_X}} } \ = \ 0
\]
when $i > 0$. Therefore
\[P_m(X) \ = \ \chi \Big( X, \OO_X(mK_X) \otimes \MI{\alinser{(m-1)K_X}}\Big), \]
and similarly
\[
P_m(Y) \ = \ \chi \Big( Y, \OO_Y(mK_Y) \otimes \MI{\alinser{(m-1)K_Y}}\Big).\]
On the other hand, 
\begin{align*}
f^* K_X \ & \lin \  K_Y \\ f^* \MI{X, \alinser{(m-1)K_X}} \ &= \  \MI{Y, \alinser{(m-1)K_Y}}
\end{align*}
thanks to \ref{Properties.AMI.Thm} (v). The Theorem then follows from the fact that Euler characteristics are multiplicative under \'etale covers. \end{proof}

\begin{remark}
It was suggested in \cite[Example 11.2.26]{PAG} that a similar argument handles adjoint bundles of the type $\OO_X(K_X + mL)$. However this -- as well as the reference given there -- is erroneous. \qed \end{remark}

\begin{exercise}
Assume that $X$ is of general type. Then for $m \ge 1$ and $i > 0$ the maps
\[
\HH{0}{X}{\OO_X(mK_X)} \otimes \HH{i}{X}{\OO_X(K_X)} \lra \HH{i}{X}{\OO_X((m+1)K_X)} \]
defined by cup product are zero. (Argue as in the proof of Koll\'ar's theorem that the map factors through $\HH{i}{X}{\OO_X((m+1)K_X \otimes \MI{\alinser{mK_X}}}$.) \qed  \end{exercise}

\subsection*{A Comparison Theorem for Symbolic Powers}

We start with a statement that follows formally from the subadditivity theorem in the form of Exercise \ref{Property.GSI}. It shows that if a graded system of ideals has any non-trivial multiplier ideal, then that system  must grow like the power of an ideal.
\begin{proposition}
Let $\fra\bull$ be a graded family of ideals, and fix an index $\ell$. Then for any $m$:
\[
\fra_\ell^m \ \subseteq \ \fra_{\ell m} \ \subseteq \ \MI{\fra\bull^{\ell m}} \ \subseteq \ \MI{\fra\bull^\ell}^m. 
\]
In particular, if $\MI{\fra\bull^\ell} \subseteq \frb$ for some ideal $\frb$, then
\[  \fra_{\ell m}\ \subseteq \ \frb^m\]
for every $m \ge 0$. \qed
\end{proposition}

In spite of the rather formal nature of this result, it has a surprising application to symbolic powers. Specifically, recall from Example \ref{symbolic.power.example} that if $Z \subseteq X$ is a reduced subscheme with ideal $\frq \subseteq \OO_X$, then the symbolic powers of $\frq$ are defined to be:
\[
\frq^{(m)} \ = \ \big \{ f \in \OO_X \ \big | \ \ord_x(f) \ge m \ \text{ for general } x \in Z \big \}.
\]
Clearly $\frq^m \subseteq \frq^{(m)}$, and if $Z$ is non-singular then equality holds. However in general the inclusion is strict:
\begin{example}
Let $Z \subseteq \CC^3 = X$ be the union of the three coordinate axes, so that 
\[  \frq \ = \ ( xy \, , \, yz \, , \, xz ) \ \subseteq \ \CC[x,y,z]. \]
Then $xyz \in \frq^{(2)}$, but $xyz \not \in \frq^2$.
\end{example}

A result of Swanson  \cite{Swan} (holding in much greater algebraic generality) states that there is an integer $k = k(Z)$ with the property that 
\[  \frq^{(km)} \ \subseteq \ \frq^m\]
for every $m$. It is natural to suppose that $k(Z)$ measures in some way the singularities of $Z$, but in fact it was established in \cite{ELS} that there is a uniform result.
\begin{theorem}
Assume that $Z$ has pure codimension $e$ in $X$. Then
\[  \frq^{(em)} \ \subseteq \ \frq^m \]
for every $m$. In particular, for any reduced $Z$ one has
\[ \frq^{(dm)} \ \subseteq \ \frq^m \]
for every $m$, where as usual $d = \dim X$. 
\end{theorem}
\begin{proof}
Consider the graded system $\frq_{(\bullet)}$ of symbolic powers. Thanks to the previous Proposition, it suffices to show that 
\[  \MI{\frq_{(\bullet)}^e } \ \subseteq \ \frq. \]
As $\frq$ is radical, membership in $\frq$ is tested at a general point of (each component of) $Z$, so we can assume that $Z$ is non-singular. But in this case 
\[  \MI{\frq_{(\bullet)}^e } \ = \ \MI{\frq^e} \ = \ \frq, \]
as required.
\end{proof}

\begin{example}
Let $T \subseteq \PP^2$ be a finite set, viewed as a reduced scheme, and denote by $I \subseteq \CC[X,Y,Z]$ the homogeneous ideal of $T$. Let $F$ be a form with the property that
\[  \mult_x(F) \ \ge \ 2m \ \ \text{for all} \ x \in T. \]
Then $F \in I^m$. (Apply the Theorem to the affine cone over $T$.) \qed
\end{example}


\section[Extension theorems]{Extension Theorems and Deformation Invariance of Plurigenera}

The most important recent applications of multiplier ideals have been to prove extension theorems. Originating in Siu's proof of the deformation invariance of plurigenera,  extension theorems are what opened the door to the spectacular progress in the minimal model program. Here we will content ourselves with a very simple result of this type, essentially the one appearing in  Siu's original article \cite{Siu} (see also \cite{Kawamata1}, \cite{Kawamata2}). As we will see,  the use of adjoint ideals renders the proof very transparent.\footnote{The utility of adjoint ideals in connection with extension theorems became clear in the work \cite{HM}, \cite{Tak} of Hacon-McKernan and Takayama on boundedness of pluricanonical mappings. The theory is further developed in \cite{HM2} and \cite{Ein.Popa}.}

The statement for which we aim is the following.
\begin{theorem} \label{Basic.Extension.Theorem}
Let $X$ be a smooth projective variety, and $S \subseteq X$ a smooth irreducible divisor. Set
\[  L \ = \ K_X \, + \, S \, + \, B, \]
where $B$ is any nef divisor, and assume that $L$ is big and that $S \not \subseteq \BBB_+(L)$. Then for every $m \ge 2$ the restriction map
\[ \HH{0}{X}{\OO_X(mL)} \lra \HH{0}{S}{\OO_S(mL_S) } \]
is surjective. 
\end{theorem}
\noi Recall that by definition $\BBB_+(L)$ denotes the stable base-locus of the divisor $kL - A$ for $A$ ample and $k \gg 0$, this being independent of $A$ provided that $k$ is sufficiently large. The hypothesis that $S \not \subseteq \BBB_+(L)$ guarantees in particular that $L_S$ is a big divisor on $S$. 

Before turning to the proof of Theorem \ref{Basic.Extension.Theorem}, let us see how it implies Siu's theorem on plurigenera in the general type case:
\begin{corollary}[Siu's Theorem on Plurigenera]
\label{Siu.Thm.Plurigenera}
Let 
\[ \pi : Y \lra T\]
 be a smooth projective family of varieties of general type. Then for each $m \ge 0$, the plurigenera 
\[  P_m(Y_t) \ =_{\textnormal{def}} \hh{0}{Y_t}{\OO_{Y_t}(mK_{Y_t})} \]
are independent of $t$. 
\end{corollary}

\begin{proof}
We may assume without loss of generality that $m \ge 2$ and that $T$ is a smooth affine curve, and we write $K_t = K_{Y_t}$.  Fixing $0 \in T$, one has $P_m(Y_0) \ge P_m(Y_t)$ for generic $t$
 by semicontinuity, so the issue is to prove the reverse inequality
 \begin{equation} \hh{0}{Y_0}{\OO_{Y_0}(mK_0)} \ \le \ \hh{0}{Y_t}{\OO_{Y_t}(mK_t)}\tag{*} 
\end{equation} for $t \in T$ in a neighborhood of $0$. For this, consider the sheaf $\pi_* \OO_Y(mK_{Y/T})$ on $T$. This is a torsion-free (and hence locally free) sheaf, whose rank computes the generic value of the $m^\text{th}$-plurigenus $P_m(Y_t)$. Moreover, the fibre of $\pi_* \OO_Y(mK_{Y/T})$ at $0$ consists of those pluri-canonical forms on $Y_0$ that extend (over a neighborhood of $0$) to forms on $Y$ itself. So to prove (*), it suffices to show that any $\eta \in \HH{0}{Y_0}{\OO_{Y_0}(mK_0)}$ extends (after possibly shrinking $T$) to some
\[  \tilde \eta \ \in \ \HH{0}{Y}{\OO_{Y}(mK_{Y/T})}. \]
We will deduce this from Theorem \ref{Basic.Extension.Theorem}. 

Specifically, we start by completing $\pi$ to a morphism 
\[ \overline \pi : \overline Y \lra \overline T\]
 of smooth projective varieties, where  $T \subseteq \overline T$ and $Y = \overline{\pi}^{-1}(T)$.  We view $Y_0 \subseteq \overline Y$ as a smooth divisor on $\overline Y$.  Let $A$ be a very ample divisor on $\overline T$, let $B = \pi^*( A - K_{\overline{T}})$, and set
\[  L \ = \ K_{\overline{Y}} + Y_0 + B   \ \lin \ K_{\overline{Y}/\overline{T}} + Y_0 + \pi^* A. \]
 We can assume by taking $A$ sufficiently positive that $B$ is nef, and  we assert that we can arrange in addition that $L$ is big and  that $Y_0 \not \subseteq \BBB_+(L)$. For the first point, it is enough to show that if $D$ is an ample divisor on $\overline{Y}$, and if $A$ is sufficiently positive, then $kL - D$ is effective for some $k \gg 0$. 
 To this end, since $Y_t$ is of general type for general $t$, we can choose $k$ sufficiently large so that $kK_{\overline{Y}/\overline{T}}  - D$ is effective on a very general fibre of $\overline \pi$.  By taking $A$ very positive, we can then guarantee that
 \[
\overline{\pi}_* \OO_{\overline Y}(kL - D) \ = \ \overline{\pi}_*  \OO_{\overline Y}(kK_{\overline{Y}/\overline{T}}- D) \otimes \OO_{\overline{T}}(kA + k0)  
 \]
 is non-zero and globally generated, which implies that $\hh{0}{\overline{Y}}{\OO_{\overline{Y}}(kL - D)} \ne 0$. For the assertion concerning $\BBB_+$, observe first that since $\linser{rY_0} = \overline \pi^*\linser{r \cdot 0}$ is free for $r \gg 0$, the divisor $Y_0 $ cannot not lie in the base locus of $\linser{kL - D + qY_0}$ for arbitrarily large $q$. On the other hand, by making $A$ more positive, we are free to replace $L$ by $L + pY_0$, i.e. we may suppose that $Y_0 \not \subseteq \BBB_+(L)$, as claimed.

 We may therefore apply Theorem \ref{Basic.Extension.Theorem} with $X = \overline{Y}$ and $S = Y_0$ to conclude that the restriction map
\begin{equation}
\HH{0}{\overline{Y}}{\OO_{\overline{Y}}(mL)} \lra \HH{0}{Y_0}{\OO_{Y_0}(mL)} 
\tag{*} \end{equation}
is surjective for every $m \ge 2$.    So  provided that we take $T$ sufficiently small so that $\OO_{\overline{T}}(A+0)|T$ is trivial, then 
\[ \OO_Y(mL) \ \cong\ \OO_Y(mK_{Y/T}), \] and surjectivity of (*) implies the surjectivity of
\[
\HH{0}{Y}{ \OO_Y(mK_{Y/T}) } \lra \HH{0}{Y_0}{\OO_{Y_0}(mK_0)},
\] 
as required.
\end{proof}

The proof of Theorem \ref{Basic.Extension.Theorem} basically follows \cite{Siu} (and \cite{PAG}), except that as we have mentioned the use of adjoint ideals substantially clarifies the presentation. Two pieces of  vocabulary  will be useful. First, if  $A$ is a divisor on a projective variety $V$, we denote by $\bs{A} \ \subseteq \OO_V$ the base ideal of the complete linear series determined by $A$. Secondly, given an ideal $\fra \subseteq \OO_V$, we say that a section $s \in \HH{0}{V}{\OO_V(A)}$ \textit{vanishes along} $\fra$ if it lies in the image of the natural map $\HH{0}{V}{\OO_V(A) \otimes \fra} \lra \HH{0}{V}{\OO_V(A)}$.

We also recall from the previous lecture a couple of facts about the asymptotic multiplier ideals associated to $L$ and its restriction to $S$:
\begin{remark} Assume that $L$ is big. Then:
\begin{enumerate}
\item[(i).]  For every $m \ge 0$,
\[  \MI{X,  \alinser{(m+1)L}} \ \subseteq \ \MI{X, \alinser{mL}}  \ \ ,  \ \ \MI{S,  \alinser{(m+1)L}_S} \ \subseteq \ \MI{S, \alinser{mL}_S}\]
\item[(ii).] (Nadel vanishing).   If $P$ is nef, then 
\[ \HHH{i}{X}{\OO_X(K_X + mL + P) \otimes \MI{X, \alinser{mL}}} \ =\ 0 \ \ \text{for } i > 0 \text{ and } m \ge 1. \]
\item[(iii).] (Subadditivity).  For any $m, q > 0$, 
\[  \MI{X, \alinser{mqL}} \ \subseteq \ \MI{X, \alinser{mL}}^q \ \ , \ \  \MI{S, \alinser{mqL}_S} \ \subseteq \ \MI{S, \alinser{mL}_S}^q. \qed \]
\end{enumerate}
\end{remark}

Returning to the situation and notation of the Theorem, 
 fix $m\ge 2$. Our analysis of extension questions revolves around the adjoint exact sequence:
  \small
\begin{equation} \label{Adj.Seq.in.Pf}
0 \to \MI{X, \alinser{(m-1)L} }\otimes \OO_X(-S) \to \Adj{S}{X, \alinser{(m-1)L}} \to\MI{S, \alinser{(m-1)L}_S} \to 0. 
\end{equation}
\normalsize
We summarize what we will use in 
\begin{lemma} \label{Extn.Thm.Lemma}
\begin{enumerate}
\item[\tn{(}i\tn{)}.] In order to prove the Theorem, it suffices to establish the inclusion
\begin{equation} \bs{ mL_S} \ \subseteq \ \MI{S, \alinser{(m-1)L}_S}.\notag
\end{equation}
\item[\tn{(}ii\tn{)}.] If $\fra \subseteq \MI{S, \alinser{(m-1)L}_S} $ is an ideal such that $\OO_S(mL_S) \otimes \fra$ is globally generated, then 
\begin{equation}
\fra\ \subseteq \ \MI{S, \alinser{mL}_S}. \notag 
\end{equation}
\end{enumerate}
\end{lemma}

\begin{proof}
For both statements, we twist through in \eqref{Adj.Seq.in.Pf} by $\OO_X(mL)$. Noting that 
\[
mL - S \lin (m-1)L + K_X + B,
\]
it follows from Nadel vanishing and the hypotheses that 
\[  \HHH{1}{X}{\OO_X(mL -S) \otimes \MI{X, \alinser{(m-1)L} }} \ = \ 0 \]
provided that $m \ge 2$. (Note that this is where it is important that we use asyptotic multiplier ideals: we do not assume that $B$ is more than nef, and so there is no ``excess positivity" in the required vanishing.) Therefore the exact sequence  \eqref{Adj.Seq.in.Pf} gives the surjectivity of the map
\[
\HHH{0}{X}{\OO_X(mL) \otimes \tn{Adj}_S(\alinser{(m-1)L}} \lra \HHH{0}{S}{\OO_S(mL_S) \otimes \MI{\alinser{(m-1)L}_S}}.
\]
The group on the left being a subspace of $\HH{0}{X}{\OO_X(mL)}$, this means  that any section of $\OO_S(mL_S)$ vanishing along $\MI{S, \alinser{(m-1)L}_S}$ lifts to a section of $\OO_X(mL)$. So if we know the inclusion in (i),  this implies that all sections of $\OO_S(mL_S)$ lift to $X$.

 For (ii), remark that  if $M$ is a (Cartier) divisor on a projective variety $V$, and if $\fra, \frb \subseteq \OO_V$ are ideals such that $\OO_V(M) \otimes \fra$ is globally generated, then $\fra \subseteq \frb$ if and only if
\[  \HH{0}{V}{\OO_V(M) \otimes \fra} \ \subseteq \ \HH{0}{V}{\OO_V(M) \otimes \frb)},\] both sides being viewed as subspaces of $\HH{0}{V}{\OO_V(M)}$. So in our situation it is enough to show that 
\begin{equation}
\HHH{0}{S}{\OO_S(mL_S) \otimes \fra)} \ \subseteq \ \HHH{0}{S}{\OO_S(mL_S) \otimes \MI{S, \alinser{mL}_S}}.  \notag
\end{equation}
Suppose then that \[ \overline t \ \in\ \HH{0}{S}{\OO_S(mL_S)\otimes \fra}\] is a section of $\OO_S(mL_S)$ vanishing along $\fra$. Since $\fra \subseteq \MI{S, \alinser{(m-1)L}_S}$, it follows as above from \eqref{Adj.Seq.in.Pf} that $\overline t$ lifts to a section $t \in \HH{0}{X}{\OO_X(mL)}$.  By definition $t$ vanishes along $\bs{mL} \subset \OO_X$, and hence its restriction $\overline t$ vanishes along 
\[ \bs{ mL } \cdot \OO_S \ \subseteq \ \MI{S, \alinser{mL}_S}, \]
as required.
 \end{proof}

In order to apply the statement (i) of the Lemma, the essential point is a comparison between the multiplier ideals of the restricted divisor $pL_S$ and those of the restricted linear series of $pL$ from $X$ to $S$:
\begin{proposition} \label{Inductive.Inclusion.Prop}
There exist a very ample divisor $A$ on $X$, a positive integer $k_0 > 0$, and a divisor $D \in \linser{k_0 L - A}$ meeting $S$ properly, such that
\begin{equation} \label{Inductive.Inclusion.Eqn}
 \MI{S, \alinser{pL_S}} \otimes \OO_S(-D_S) \ \subseteq \ \MI{S, \alinser{(p+k_0 -1)L}_S} 
\end{equation}
for every $p \ge 0$. 
\end{proposition}

Granting this for the moment, we complete the proof of the Theorem. In fact, fix $m$, and apply equation \eqref{Inductive.Inclusion.Eqn} with $p = qm$ for $q \gg 0$.
One finds:
\begin{align*}
\bs{mL_S}^q \otimes \OO_S(-D_S)\  &\subseteq \ \bs{mqL_S} \otimes \OO_S(-D_S) \\
&\subseteq \ \MI{S,\alinser{mqL_S}} \otimes \OO_S(-D_S) \\
&\subseteq \ \MI{S, \alinser{(mq + k_0 - 1)L}_S }\\
& \subseteq \ \MI{S, \alinser{mqL}_S} \\
& \subseteq \ \MI{S, \alinser{mL}_S}^q,
\end{align*}
the last inclusion coming from the subadditivity theorem. But we assert that having the inclusion
\begin{equation}  \bs{mL_S}^q \otimes \OO_S(-D_S) \ \subseteq \ \MI{S, \alinser{mL}_S}^q\tag{+}\end{equation}
for all $q \gg 0$ forces $\bs{mL_S} \subseteq \MI{S, \alinser{mL}_S}$ and hence also the inclusion   in Lemma \ref{Extn.Thm.Lemma} (i). In fact, it follows from the construction of multiplier ideals that there are finitely many divisors $E_\alpha$ lying over $S$, together with integers $r_\alpha > 0$, such that a germ $f \in \OO_S$ lies in $\MI{S, \alinser{mL}_S}$ if and only if $\ord_{E_\alpha}(f) \ge r_\alpha$ for every $\alpha$. But (+) implies that if $f \in    \bs{mL_S}$, then
\[
q \cdot \ord_{E_\alpha}  (f ) + \ord_{E_\alpha}(D_S) \ \ge \ q \cdot r_\alpha
\]
for each $\alpha$, and letting $q \to \infty$ one finds that $\ord_{E_\alpha}   (f)  \ge r_\alpha$, as required. Thus Lemma \ref{Extn.Thm.Lemma} applies, and this completes the proof of the Theorem granting Proposition \ref{Inductive.Inclusion.Prop}.

It remains to prove the Proposition. Here the essential point is statement (ii) of the Lemma. 
 \begin{proof}[Proof of Proposition \ref{Inductive.Inclusion.Prop}]  By Nadel vanishing and Castelnuovo-Mumford regularity, we can find a very ample divisor $A$ so that for every $q \ge 0$ the sheaf \[
 \OO_S(qL_S + A_S) \otimes \MI{S, \alinser{qL}_S}
 \]
 is globally generated (cf Exercise \ref{Asympt.Univ.Gl.Gen.Ex}).  Next, since $S \not \subseteq \BBB_+(L)$, for $k_0 \gg 0$ we can take a divisor
$ D \in \linser{k_0L - A}$ that does not contain $S$. We will show by induction on $p$ that \eqref{Inductive.Inclusion.Eqn}
 holds with these choices of the data. 
 
 For $p = 0$ the required inclusion holds by virtue of the fact that $D + A \lin  k_0 L$, which yields
 \[  
 \OO_S(-D_S) \ \subseteq \ \MI{S, \alinser{k_0L}_S} \ \subseteq \ \MI{S, \alinser{(k_0 -1)L}_S}.
 \]
 Assuming that \eqref{Inductive.Inclusion.Eqn}
is satisfied for a given value of $p$, we will show that it holds also for $p + 1$. To this end, observe first that
\[  \OO_S\big((p + k_0)L_S - D_S\big) \otimes \MI{S, \alinser{pL_S}} \]
is globally generated thanks to our choice of $A$. Applying Lemma \ref{Extn.Thm.Lemma} (ii) with $m = p + k_0$ and $\fra = \MI{S, \alinser{pL_S}} \otimes \OO_S(-D_S)$, it follows using the inductive hypothesis 
\[
 \MI{S, \alinser{pL_S}} \otimes \OO_S(-D_S) \ \subseteq \ \MI{S, \alinser{(p+k_0 -1)L}_S} 
\]
 that in fact \[ \MI{S, \alinser{pL_S}} \otimes \OO_S(-D_S)
\ \subseteq \   \MI{S, \alinser{(p+k_0)L}_S}.\] Therefore also
\[   \MI{S, \alinser{(p+1)L_S}} \otimes \OO_S(-D_S) \   \subseteq \
   \MI{S, \alinser{(p+k_0)L}_S} ,
\]
which completes the induction.
 \end{proof}

\end{document}